\documentclass[preprint]{article}%
\usepackage{graphicx}
\usepackage{amsmath,amsxtra,amssymb,latexsym, amscd,amsthm}
\usepackage[mathscr]{eucal}
\usepackage{floatflt}
\usepackage{bbm}
\newtheorem{thm}{Theorem}[section]
\newtheorem{cor}{Corollary}[section]
\newtheorem{lem}{Lemma}[section]
\newtheorem{prop}{Proposition}[section]
\theoremstyle{definition}

\newtheorem{assum}{Assumption}[section]
\theoremstyle{remark}
\newtheorem{rem}{Remark}[section]
\numberwithin{equation}{section}
\def\ind{{\rm 1\hspace{-0.90ex}1}}
\begin{document}

\title{Weak convergence of delay SDEs with applications to Carath\'eodory approximation}
\author{T. C. Son\thanks{Department of Mathematics, VNU University of Science, Vietnam National University, Hanoi, 334 Nguyen
Trai, Thanh Xuan, Hanoi, 084 Vietnam.} \and N. T. Dung$^\ast$\thanks{Corresponding author. Email: dung@hus.edu.vn}\and N. V. Tan\thanks{Department of Foundation, Academy of Cryptography Techniques, 141 Chien Thang, Thanh Tri, Hanoi, Vietnam} \and T. M. Cuong$^\ast$\and H. T. P. Thao$^\ast$\and P. D. Tung$^\ast$}

\date{September 4, 2021}          % Ngay
\maketitle
\begin{abstract}
In this paper, we consider a fundamental class of stochastic differential equations with time delays. Our aim is to investigate the weak convergence with respect to delay parameter of the solutions. Based on the techniques of Malliavin calculus, we obtain an explicit estimate for the rate of convergence. An application to the Carath\'eodory approximation scheme of stochastic differential equations is provided as well.
\end{abstract}
%\section{Introduction}       % Muc dau tien
%{\large Densities, Tail probabilities and }
\section{Introduction}
It is known that the Carath\'eodory approximation scheme was introduced by Carath\'eodory in the early part of 20th century for deterministic differential equations, see e.g. \cite{C.L}. In the context of stochastic equations, the first Carath\'eodory approximation results were obtained by Bell and Mohammed \cite{Bell1989} (also see Section 2.6 in \cite{Mao2008} for a general formulation). Let $x_0\in \mathbb{R},n\geq 1$ and $(B(t))_{t\in [0,T]}$ be a standard Brownian motion. We consider stochastic differential equation (SDE)
\begin{equation}\label{cara01}
x(t)=x_0+\int_0^t b(s,x(s))ds+\int_0^t \sigma(s,x(s))dB(s),\,\,\,t\in[0,T]
\end{equation}
and its Carath\'eodory approximation
\begin{equation}\label{cara02}
\left\{
  \begin{array}{l}
x^n(t)=x_0+\int_0^t b(s,x^n(s-1/n))ds+\int_0^t \sigma(s,x^n(s-1/n))dB(s),\,\,\,t\in[0,T]\\
\\
x^n(t)=x_0,\,\,\,t\in [-1/n,0].
  \end{array}
\right.
\end{equation}
It was proved in \cite{Bell1989,Mao2008} that, when the coefficients are Lipschitz and have linear growth, $x^n(t)$ strongly converges to $x(t)$ as $n\to\infty.$ Moreover, the following estimate for the strong  rate of convergence holds
\begin{equation}\label{cara03}
E\left[\sup\limits_{0\leq t\leq T}|x^n(t)-x(t)|^2\right]\leq \frac{C}{n},\,\,\,n\geq 1,
\end{equation}
where $C$ is a positive constant not depending on $n$

%Note that, comparing Picard's approximation,
As discussed in Section 2.6 of \cite{Mao2008}, the advantage of Carath\'eodory approximation is that we do not need to compute $x^1(t),\cdots,x^{n-1}(t)$ but compute $x^n(t)$ directly (this reduces a lot of calculations on stochastic integrals). In addition, the Carath\'eodory approximation also works well for SDEs with non-Lipschitz coefficients. In the last decades, the Carath\'eodory approximation has been considered for various stochastic differential equations. Among others, we mention Turo \cite{Turo} for stochastic functional differential equations,  Mao \cite{Mao1,Mao2} and Liu \cite{Liu} for semilinear stochastic evolution equations with time delays, Ferrante \& Rovira \cite{Ferrante2010} for delay differential equations driven by fractional Brownian motion, Faizullah \cite{Fai} for SDEs under G-Brownian motion, Benabdallah \& Bourza \cite{BeBo} for perturbed SDEs with reflecting boundary, Mao et al. \cite{WMao2018} for doubly perturbed SDEs, etc.

It is also known that the weak convergence rate of numerical approximations is very useful in practical applications (see, e.g. \cite{Bally1996} for a short discussion). In fact, for certain numerical schemes such as Milstein scheme, Runge-Kutta scheme and Euler-Maruyama scheme, etc. many weak convergence results can be found in the literature, see e.g. \cite{Bally1996,Buckwar2008,Jourdain2011,Kloeden1992}. However, to the best of our knowledge, the weak convergence results of  the Carath\'eodory approximation are scarce even for the system (\ref{cara01})-(\ref{cara02}). Motivated by this observation, our aim is to partially fill up this gap.

In this paper, to make the problem more interesting, we consider delay stochastic differential equations of the form
\begin{equation}\label{cara04}
\left\{
  \begin{array}{l}
X_{\tau_1}(t)=\varphi(0)+\int_0^t b(s,X_{\tau_1}(s),X_{\tau_1}(s-\tau_1))ds+\int_0^t\sigma(s,X_{\tau_1}(s),X_{\tau_1}(s-\tau_1))dB(s),\,t\in[0,T]\\
\\
X_{\tau_1}(t)=\varphi(t),\,\,\,t\in [-\tau_1,0]
  \end{array}
\right.
\end{equation}
and
\begin{equation}\label{cara05}
\left\{
  \begin{array}{l}
X_{\tau_2}(t)=\varphi(0)+\int_0^t b(s,X_{\tau_2}(s),X_{\tau_2}(s-\tau_2))ds+\int_0^t\sigma(s,X_{\tau_2}(s),X_{\tau_2}(s-\tau_2))dB(s),\,t\in[0,T]\\
\\
X_{\tau_2}(t)=\varphi(t),\,\,\,t\in [-\tau_2,0],
  \end{array}
\right.
\end{equation}
where $0\leq \tau_1, \tau_2<\infty$ and $\varphi: (-\infty, 0]\to \mathbb{R}$ is a bounded deterministic function. Our aim is to study the weak convergence of $X_{\tau_2}(t)$ to $X_{\tau_1}(t)$ as $\tau_2\to\tau_1.$ More specifically, we will employ the techniques of Malliavin calculus to provide a quantitative estimate for the quantity
$$|Eg(X_{\tau_{2}}(t))-Eg(X_{\tau_{1}}(t))|,$$
where $g$ is a bounded and measurable function, see Theorem \ref{them2}. Our method is different from the existing ones in the literature and based on a new result established recently in \cite{Dung2020}. The restatement of this new result and further comments  will be given in Lemma \ref{lm2.1} and Remark \ref{uifk5}.

Our results applied to the Carath\'eodory approximation system (\ref{cara01})-(\ref{cara02}) yield%, see Corollary \ref{kgld1},
\begin{equation}\label{cara03a}
\sup\limits_{0\leq t\leq T}|Eg(x^n(t))-Eg(x(t))|\leq \frac{C}{\sqrt{n}},\,\,\,n\geq 1.
\end{equation}
It should be noted that, when the test function $g$ is Lipschitz continuous, the strong rate (\ref{cara03}) implies the weak rate (\ref{cara03a}). However, for bounded and measurable test functions, there is no such implication and hence, the novelty of our results lies in the fact that (\ref{cara03a}) holds true for any bounded and measurable test function. The price to pay is that, besides Lipschitz and linear growth conditions, we have to impose additional assumptions on the coefficients, see conditions $(ii)$-$(iii)$ of Corollary \ref{kgld1}.

The rest of this article is organized as follows. In Section 2, we recall some concepts of Malliavin calculus and a general result obtained in our recent work \cite{Dung2020}. Our main results are then stated and proved in Section 3. The conclusion and some remarks are given in Section 4.
\section{Preliminaries}
As we have said in the Introduction, this paper is based on techniques of Malliavin calculus. For the reader's convenience, let us recall some elements of Malliavin calculus (for more details see \cite{nualartm2}). We suppose that $(B(t))_{t\in [0,T]}$ is defined on a complete probability space $(\Omega,\mathcal{F},\mathbb{F},P)$, where $\mathbb{F}=(\mathcal{F}_t)_{t\in [0,T]}$ is a natural filtration generated by the Brownian motion $B.$ For $h\in L^2[0,T],$ we denote by $B(h)$ the Wiener integral
	$$B_h=\int_0^T h(t)dB(t).$$
	Let $\mathcal{S}$ denote the dense subset of $L^2(\Omega, \mathcal{F},P)$ consisting of smooth random variables of the form
	\begin{equation}\label{ro}
	F=f(B_{h_1},...,B_{h_n}),
	\end{equation}
	where $n\in \mathbb{N},h_1,...,h_n\in L^2[0,T]$ and $f$ is an infinitely
differentiable function such that together with all its partial derivatives
has at most polynomial growth order.  If $F$ has the form (\ref{ro}), we define its Malliavin derivative as the process $DF:=\{D_tF, t\in [0,T]\}$ given by
	$$D_tF=\sum\limits_{k=1}^n \frac{\partial f}{\partial x_k}(B_{h_1},...,B_{h_n}) h_k(t).$$
More generally, for each $k\geq 1,$ we can define the iterated derivative operator  by setting $D^{k}_{t_1,...,t_k}F=D_{t_1}...D_{t_k}F.$ For any $p,k\geq 1,$ we shall denote by $\mathbb{D}^{k,p}$ %\big(L^2([0,T],\mathbb{R},du)\big)$
the closure of $\mathcal{S}$ with respect to the norm
$$\|F\|^p_{k,p}:=E|F|^p+E\bigg[\int_0^T|D_{t_1}F|^pdt_1\bigg]+...+E\bigg[\int_0^T...\int_0^T|D^{k}_{t_1,...,t_k}F|^pdt_1...dt_k\bigg].$$
A random variable $F$ is said to be Malliavin differentiable if it belongs to $\mathbb{D}^{1,2}.$ For the convenience of the reader, we recall that  the derivative
	operator $D$ satisfy the chain rule, i.e.
\begin{equation}\label{haha1}
D\varphi(F)=\varphi'(F) DF.
\end{equation}
 Moreover, we have the following relations between Malliavin derivative and the integrals
$$D_r\left(\int_0^Tu(s)ds\right)=\int_r^TD_ru(s)ds,$$
$$D_r\left(\int_0^Tu(s)dB(s)\right)=u_r+\int_r^TD_ru(s)dB(s)$$
for all $0\leq r\leq T,$ where $(u(t))_{t\in [0,T]}$ is an $\mathbb{F}$-adapted and Malliavin differentiable stochastic process.

An important operator in the Malliavin's calculus theory is the divergence operator $\delta$, it is the adjoint of the derivative
	operator $D$. The domain of $\delta$ is the set of all functions $u \in L^{2}(\Omega, L^{2}[0, T])$ such that for all $F\in \mathbb{D}^{1,2}$ it holds that
	$$E|\langle DF, u \rangle_{L^{2}[0,T]}|\leq C(u)\|F\|_{L^{2}[0,T]},$$
	where $C(u)$ is some positive constant depending on $u$. In particular, if $u \in dom\delta$, then $\delta(u)$ is characterized by following duality relationship
\begin{equation}\label{hha}
E\langle DF, u\rangle_{L^{2}[0,T]}=E[F\delta(u)] \text{ for any } F\in \mathbb{D}^{1,2}.
\end{equation}
Let $\mathcal{B}$ denote the space of measurable functions $g:\mathbb{R}\to \mathbb{R}$ such that $\|g\|_\infty:=\sup\limits_{x\in \mathbb{R}}|g(x)|\leq 1.$ We have the following.
\begin{lem}\label{lm2.1}
Let $F_1 \in \mathbb{D}^{2,2}$ satisfy that  $\|DF_1\|_{L^2[0,T]}>0\,\,a.s.$ Then, for any random variable $F_2\in \mathbb{D}^{1,2}$ and any $g\in \mathcal{B},$  we have
\begin{align}
&|Eg(F_1)-Eg(F_2)|\notag\\
&\leq C \left(E\|DF_1\|^{-8}_{L^2[0,T]}E\left(\int_0^T\int_0^T|D_\theta D_rF_1|^2d\theta dr\right)^2+(E\|DF_1\|^{-2}_{L^2[0,T]})^2\right)^{\frac{1}{4}}\|F_1-F_2\|_{1,2},\notag%\label{uu4}
\end{align}
provided that the expectations exist, where $C$ is an absolute constant.
\end{lem}
\begin{proof} This lemma comes from Theorem 3.1 in our recent paper \cite{Dung2020}. Here, for the reader's convenience, we recall its proof. We write $\langle . , . \rangle$  instead of $\langle . , . \rangle_{L^2[0,T]}$ and $\|.\|$ instead of $\|.\|_{L^2[0,T]}.$ By the routine approximation argument, we can assume that $g$ is continuous. Indeed, for example, we can approximate $g$ by bounded and Lipschitz continuous functions defined by
$$g_\varepsilon(x)=\int_{-\infty}^\infty \ind_{\{|y|<\frac{1}{\varepsilon}\}}g(y)\rho_\varepsilon(x-y)dy,\,\,\varepsilon>0,$$
where $\rho_\varepsilon$ is the standard mollifier:  $\rho_\varepsilon(x)=\frac{\rho(x/\varepsilon)}{\varepsilon},$ where $\rho(x)=C\ind_{\{|x|<1\}}e^{\frac{1}{x^2-1}}$ and $C$ is a constant such that $\int_{-\infty}^\infty\rho(x)dx=1.$

The function $\psi(y):=\int_{-\infty}^yg(z)dz$ is differentiable with bounded derivative. Hence, $\psi(F_i)\in\mathbb{D}^{1,2}$ for $i=1,2.$ We obtain
$$\langle D\psi(F_i) , DF_1 \rangle=g(F_i) \langle DF_i , DF_1 \rangle, \,\,\, i=1,2,$$
or equivalently
$$\langle D\int_{-\infty}^{F_i} g(z)dz,DF_1  \rangle=g(F_i) \langle DF_i , DF_1  \rangle,\,\,\, i=1,2.$$
As a consequence, we get
\begin{align*}
\langle &D\int_{F_2}^{F_1} g(z)dz, DF_1 \rangle=g(F_1) \langle DF_1, DF_1 \rangle -g(F_2) \langle DF_2 , DF_1 \rangle
\\& =\left(g(F_1)-g(F_2)\right) \langle DF_1 , DF_1 \rangle +g(F_2) \langle DF_1-DF_2 , DF_1 \rangle,
\end{align*}
and hence,
$$g(F_1)-g(F_2)=\dfrac{\langle D\int_{F_2}^{F_1} g(z)dz, DF_1 \rangle}{\|DF_1\|^2}-\dfrac{g(F_2) \langle DF_1-DF_2 , DF_1\rangle}{\|DF_1\|^2}.$$
Taking the expectation yields
\begin{align}
Eg(F_1)-Eg(F_2)&=E\left[\dfrac{\langle D\int_{F_2}^{F_1} g(z)dz, DF_1 \rangle}{\|DF_1\|^2}\right]-E\left[\dfrac{g(F_2) \langle DF_1-DF_2 , DF_1 \rangle}{\|DF_1\|^2}\right].\label{oldl1}
%\\& =E\left[\int_{F_2}^{F_1} g(z)dz\delta\left(\dfrac{DF_1 }{\|DF_1\|^2}\right)\right]-E\left[\dfrac{g(F_2) \langle DF_1-DF_2 , DF_1 \rangle}{\|DF_1\|^2}\right],
\end{align}
 Now we consider the stochastic process $$u_r:=\dfrac{D_rF_1 }{\|DF_1\|^2}, \ \ 0\leq r\leq T$$
and use Proposition 1.3.1 in \cite{nualartm2} to get
\begin{align}%E\left[\delta\left(\dfrac{DF_1}{\|DF_1\|^2}\right)^2\right]=
E[\delta(u)^2]&\leq \int_0^TE|u_r|^2dr+\int_0^T\int_0^TE|D_\theta u_r|^2d\theta dr\notag\\
&=E\|D F_1\|^{-2}+\int_0^T\int_0^TE|D_\theta u_r|^2d\theta dr\label{uu2}
\end{align}
 By the chain rule (\ref{haha1}) for Malliavin derivatives, we have
 $$D_\theta u_r=\dfrac{D_\theta D_rF_1}{\|DF_1\|^2}-2\dfrac{D_rF_1\langle D_\theta DF_1,DF_1\rangle}{\|DF_1\|^4}, \ \ 0\leq\theta, r\leq T.$$
By  the Cauchy-Schwarz inequality, we obtain
  \begin{align}
\int_0^T\int_0^TE|D_\theta u_r|^2d\theta dr
 &\leq2E\left[\dfrac{\int_0^T\int_0^T|D_\theta D_rF_1|^2d\theta dr}{\|DF_1\|^4}\right]+8E\int_0^T\int_0^T\dfrac{|D_rF_1|^2\|D_\theta DF_1\|^2}{\|DF_1\|^6}d\theta dr\notag\\
 &=2E\left[\dfrac{\int_0^T\int_0^T|D_\theta D_rF_1|^2d\theta dr}{\|DF_1\|^4}\right]+8E\left[\dfrac{\int_0^T\int_0^T|D_\theta D_rF_1|^2d\theta dr}{\|DF_1\|^4}\right]\notag\\
 & \leq 10\left(E\left[\left(\int_0^T\int_0^T|D_\theta D_rF_1|^2d\theta dr\right)^2\right]\right)^{\frac{1}{2}}\left(E\left[\|DF_1\|^{-8}\right]\right)^{\frac{1}{2}}.\label{uu3}
% & \leq 10T\left(\int_0^T\int_0^TE|D_\theta D_rF_1|^4d\theta dr\right)^{\frac{1}{2}}\left(E\left[\|DF_1\|^{-8}\right]\right)^{\frac{1}{2}}.
%\\& \leq C  \left(E\left[\left(\int_0^T\int_0^T|D_\theta D_rF_1|^2d\theta dr\right)^2\right]\right)^{\frac{1}{2}}
 \end{align}
Inserting (\ref{uu3}) into (\ref{uu2}) yields
\begin{equation}\label{uifk3}
E[\delta(u)^2]\leq E\|D F_1\|^{-2}+10\left(E\left[\left(\int_0^T\int_0^T|D_\theta D_rF_1|^2d\theta dr\right)^2\right]\right)^{\frac{1}{2}}\left(E\left[\|DF_1\|^{-8}\right]\right)^{\frac{1}{2}}.
\end{equation}
By using the relation (\ref{hha}),  the Cauchy-Schwarz  inequality and $\|g\|_\infty\leq 1,$ it follows from (\ref{oldl1}) that
\begin{align}
&|Eg(F_1)-Eg(F_2)|\leq E\left|\int_{F_2}^{F_1} g(z)dz\delta\left(u\right)\right|+E\left|\dfrac{g(F_2) \langle DF_1-DF_2 , DF_1 \rangle}{\|DF_1\|^2}\right|\notag\\
&\leq\left(E|F_1-F_2|^2\right)^{\frac{1}{2}} \left(E[\delta\left(u\right)^2]\right)^{\frac{1}{2}}+E\left|\dfrac{ \langle DF_1-DF_2 , DF_1 \rangle}{\|DF_1\|^2}\right|\notag\\
&\leq\left(E|F_1-F_2|^2\right)^{\frac{1}{2}} \left(E[\delta\left(u\right)^2]\right)^{\frac{1}{2}}+E\left[\dfrac{\| DF_1-DF_2\|}{\|DF_1\|}\right]\notag\\
&\leq\left(E|F_1-F_2|^2\right)^{\frac{1}{2}} \left(E[\delta\left(u\right)^2]\right)^{\frac{1}{2}}+\left(E\|DF_1-DF_2\|^2\right)^{\frac{1}{2}} \left(E\|DF_1\|^{-2}\right)^{\frac{1}{2}}.\label{uu1}
\end{align}
So we can obtain the desired conclusion by inserting (\ref{uifk3}) into (\ref{uu1}) and then using Cauchy-Schwarz inequality.
\end{proof}
\begin{rem}\label{uifk5} Lemma \ref{lm2.1} shows that the weak convergence can be controlled by the norm $\|.\|_{1,2}$ in the space of Malliavin differentiable random variables. In its statement, we do not require any special structure of $F_1$ and $F_2.$ We therefore believe that Lemma \ref{lm2.1} can be applied to other numerical schemes such as Euler-Maruyama and Runge-Kutta schemes. In the present paper, the main reason for choosing Carath\'eodory scheme is due to the lack of weak convergence results for this scheme in the literature. It is also worth mentioning that the proof of Lemma \ref{lm2.1} heavily relies on dimension one. The generalization to higher dimensions will be a difficult and interesting problem.
\end{rem}

\section{The main results}
In the whole paper, we are going to impose the following assumptions:
\begin{assum}\label{assum1} The functions $b,\sigma:[0,T]\times \mathbb{R}^2\to \mathbb{R}$ are Lipschitz and have linear growth, i.e. there exists $L>0$ such that
	$$|b(t,x_1,y_1)-b(t,x_2,y_2)|+|\sigma(t,x_1,y_1)-\sigma(t,x_2,y_2)|\leq L(|x_1-x_2|+|y_1-y_2|)$$
for all $t\in [0,T],x_1,x_2,y_1,y_2\in \mathbb{R}$ and
	$$|b(t,x,y)|+|\sigma(t,x,y)|\leq L(|x|+|y|+1),\,\,\forall\,t\in [0,T],x,y\in \mathbb{R}.$$
\end{assum}
\begin{assum}\label{assum5}
For every $t\in [0,T],$ $b(t,.,.), \sigma(t,.,.)$ are twice differentiable functions with the partial derivatives bounded by $L$.
\end{assum}

Hereafter, we denote by $C$ (with or without an index) a generic constant which may vary at each appearance. For any function $h$ of $n$ variables, we denote
$$h'_i(x_1,\cdots,x_n):=\frac{\partial h}{\partial x_i}(x_1,\cdots,x_n),\quad h''_{ij}(x_1,\cdots,x_n):=\frac{\partial h}{\partial x_i\partial x_j}(x_1,\cdots,x_n).$$
For any $a,b\in \mathbb{R},$ we denote $a\vee b=\max\{a,b\}.$ In the proofs, we frequently use the fundamental inequality $$(a_1+\cdots+a_n)^p\leq n^{p-1}(a_1^p+\cdots+a_n^p)$$ for all $a_1,...,a_n\geq 0$ and $p\geq 2.$

\subsection{Malliavin derivative of delay SDEs}
In this Subsection, given $\tau\geq 0,$ we investigate Malliavin derivatives of the solution to the following delay SDE
\begin{equation}\label{eq2}
\left\{
\begin{array}{l}
		X_{\tau}(t)=\varphi(0)+\int_0^t b(s,X_{\tau}(s),X_{\tau}(s-\tau))ds+\int_0^t\sigma(s,X_{\tau}(s),X_{\tau}(s-\tau))dB(s),\,\,\,t\in [0,T],\\
\\
		X_{\tau}(t) = \varphi(t), \quad t\in [-\tau, 0].
\end{array}\right.
\end{equation}
For the Malliavin differentiability of the solutions, we have the following.
\begin{prop}\label{prop1}
Let  Assumption \ref{assum1} hold. Then, the equation (\ref{eq2})  has a unique solution $(X_{\tau}(t))_{t\in[-\tau,T]}$ and this solution is Malliavin differentiable. Moreover, the derivative $D_{\theta}X_{\tau}(t)$ satisfies

\noindent (i) When $t\in [-\tau,0],$ $D_{\theta}X_{\tau}(t) = 0$  for all $0\leq \theta \leq T,$

\noindent (ii) When $t\in (0,T],$  $D_{\theta}X_{\tau}(t) = 0$  for $\theta >t$ and
	\begin{align}
		D_\theta X_{\tau}(t)&=\sigma(\theta,X_{\tau}(\theta),X_{\tau}(\theta-\tau))
		+\int_{\theta}^{t} \bar{b}^{\tau}_{2}(s)D_\theta X_{\tau}(s)ds+\int_{\theta+\tau}^{t} \bar{b}^{\tau}_{3}(s)D_\theta X_{\tau}(s-\tau)ds\nonumber\\
		&+\int_{\theta}^{t}\bar{\sigma}^{\tau}_{2}(s)D_\theta X_{\tau}(s)dB(s)+\int_{\theta+\tau}^{t}\bar{\sigma}^{\tau}_{3}(s)D_\theta X_{\tau}(s-\tau)dB(s),\,\,\,0\leq \theta\leq t-\tau,\label{mall1}
	\end{align}
	\begin{align}
		D_\theta X_{\tau}(t)=\sigma(\theta,X_{\tau}(\theta),X_{\tau}(\theta-\tau))
		+\int_{\theta}^{t} \bar{b}^{\tau}_{2}(s)D_\theta X_{\tau}(s)ds+\int_{\theta}^{t}\bar{\sigma}^{\tau}_{2}(s)D_\theta X_{\tau}(s)dB(s),\,\,\,(t-\tau)\vee 0< \theta\leq t,\label{mall2}
	\end{align}
	where  $\bar{b}^{\tau}_{2}(s)$, $\bar{b}^{\tau}_{3}(s)$, $\bar{\sigma}^{\tau}_{2}(s)$, $\bar{\sigma}^{\tau}_{3}(s)$ are adapted stochastic processes and bounded by the Lipschitz constant $L.$ Here, we use the convention $[0,t-\tau]=\emptyset$ if $t<\tau.$
\end{prop}

\begin{proof}
	For $t\in[-\tau, 0]$, we have $X_{\tau}(t) = \varphi(t)$ is deterministic. Hence, the Malliavin derivative of $X_{\tau}(t)$ always vanishes. On the other hand, since the solution $(X_{\tau}(t))_{t\in [0,T]}$ is $\mathbb{F}$-adapted, we always have $D_{\theta}X_{\tau}(t)= 0$ for $\theta>t$. When $\theta\leq t$,  by using the same argument as in the proof of Theorem 2.2.1 in \cite{nualartm2}, we can show that the solution $(X_{\tau}(t))_{t\in [0,T]}$ is also Malliavin differentiable. Applying the operator $D$ to the equation \eqref{eq2} we deduce
\begin{align}
D_{\theta}X_{\tau}(t) &= \sigma(\theta, X_{\tau}(\theta),X_{\tau}(\theta-\tau))\notag\\
&+ \int_{\theta}^{t}D_{\theta}[b(s,X_{\tau}(s),X_{\tau}(s-\tau))]ds+\int_{\theta}^{t}D_{\theta}[\sigma(s,X_{\tau}(s),X_{\tau}(s-\tau))]dB_{s},\,\,\,t\in (0,T].\label{prmall1}
\end{align}
By Proposition 1.2.4 in \cite{nualartm2} and Lipchitz property of $b$ and $\sigma$, there exist adapted processes $\bar{b}^{\tau}_{2}(s)$, $\bar{b}^{\tau}_{3}(s)$, $\bar{\sigma}^{\tau}_{2}(s)$ and $\bar{\sigma}^{\tau}_{3}(s)$, uniformly bounded by $L$ such  that
\begin{equation}\label{prmall1.1}
		D_{\theta}\left[b(s,X_{\tau}(s),X_{\tau}(s-\tau))\right] =  \bar{b}^{\tau}_{2}(s)D_\theta X_{\tau}(s)+ \bar{b}^{\tau}_{3}(s)D_{\theta} X_{\tau}(s-\tau),
	\end{equation}
\begin{equation}\label{prmall1.2}
D_{\theta}\left[\sigma(s,X_{\tau}(s),X_{\tau}(s-\tau))\right] = \bar{\sigma}^{\tau}_{2}(s)D_\theta X_{\tau}(s)+ \bar{\sigma}^{\tau}_{3}(s)D_{\theta} X_{\tau}(s-\tau).
	\end{equation}
Because $D_{\theta} X_{\tau}(s-\tau)=0$ for $\theta>s-\tau,$ we obtain the equations (\ref{mall1}) and (\ref{mall2}) by inserting (\ref{prmall1.1}) and (\ref{prmall1.2}) into (\ref{prmall1}). So the proof of Proposition is complete.
\end{proof}	
\begin{rem}\label{rem1}
	If $b(t,\cdot,\cdot)$ and $\sigma(t,\cdot,\cdot)$ are continuously differentiable, then $\bar{b}^{\tau}_{2}(s) =b_2'(s,X_{\tau}(s),X_{\tau}(s-\tau))$, $\bar{b}^{\tau}_{3}(s) =b_3'(s,X_{\tau}(s),X_{\tau}(s-\tau))$, $\bar{\sigma}^{\tau}_{2}(s) =\sigma_2'(s,X_{\tau}(s),X_{\tau}(s-\tau))$, $\bar{\sigma}^{\tau}_{3}(s) =\sigma_3'(s,X_{\tau}(s),X_{\tau}(s-\tau))$. Consequently, under Assumptions \ref{assum1} and \ref{assum5},  we have
\begin{align}\label{rem1.1}
D_\theta X_{\tau}(t)&=\sigma(\theta,X_{\tau}(\theta),X_{\tau}(\theta-\tau))
+\int_\theta^t b_2'(s,X_{\tau}(s),X_{\tau}(s-\tau))D_\theta X_{\tau}(s)ds\notag\\
&+\int_{\theta+\tau}^t b_3'(s,X_{\tau}(s),X_{\tau}(s-\tau))D_\theta X_{\tau}(s-\tau)ds+\int_\theta^t\sigma_2'(s,X_{\tau}(s),X_{\tau}(s-\tau))D_\theta X_{\tau}(s)dB(s)\notag\\
&+\int_{\theta+\tau}^t\sigma_3'(s,X_{\tau}(s),X_{\tau}(s-\tau))D_\theta X_{\tau}(s-\tau)dB(s),\,\,0\leq \theta\leq t-\tau
\end{align}
and
\begin{align}\label{rem1.2}
D_\theta X_{\tau}(t)&=\sigma(\theta,X_{\tau}(\theta),X_{\tau}(\theta-\tau))+\int_\theta^t b_2'(s,X_{\tau}(s),X_{\tau}(s-\tau))D_\theta X_{\tau}(s)ds\notag\\
&+\int_\theta^t\sigma_2'(s,X_{\tau}(s),X_{\tau}(s-\tau))D_\theta X_{\tau}(s)dB(s),\,\,\,(t-\tau)\vee 0< \theta\leq t.
\end{align}
Moreover, $X_{\tau}(t)$ is twice Malliavin differentiable. The proof of this fact is similar to that of Theorem 2.2.2 in \cite{nualartm2}, we omit the details. The second order Malliavin derivative of $X_{\tau}(t)$  will be computed in Proposition \ref{9jk3k} below.
\end{rem}
%%%%%%%%%%%%%%%%%%%%%%%%%%%%%%%%%%%%%%%%%%%%%%%%%%%%%%%%%%%%%%%%%%%%
We have the following familiar estimates for the moments.
\begin{lem}\label{lem1}
Let Assumption \ref{assum1} hold. Then, for every $p\geq 2,$ we  have
\begin{equation}\label{9hjd2}
E|X_{\tau}(t)|^p\leq C_{p,L,T},\,\,\,t\in[0,T]
\end{equation}
and
\begin{equation}\label{9hjd1}
E|X_{\tau}(t)-X_{\tau}(s)|^p\leq C_{p,L,T}|t-s|^{\frac{p}{2}},\,\,\,s,t\in[0,T],
\end{equation}
where $C_{p,L,T}$ is a positive constant.
\end{lem}
\begin{proof} The proof  is similar to that of Lemmas 6.1 and 6.2 in \cite{Mao2008}. So we omit it. Here, just remark that we use the Burkholder-Davis-Gundy inequality instead of the It\^o isometry.
\end{proof}

\begin{lem}\label{lem2} Let Assumption \ref{assum1} hold. Then, for every $p\geq 2,$ we  have
\begin{equation}\label{lem2.1}
	E|D_\theta X_{\tau}(t)|^p\leq C_{p,L,T},\,\,\,0\leq \theta \leq t\leq T
\end{equation}
and
\begin{equation}\label{lem2.2}
E|D_\theta X_{\tau}(t)-D_\theta X_{\tau}(s)|^p\leq C_{p,L,T}|t-s|^{p/2},\,\,\,0\leq \theta \leq s\vee t\leq T,
\end{equation}
where $C_{p,L,T}$ is a positive constant.
\end{lem}
	\begin{proof} By the linear growth property of $\sigma$ and the estimate (\ref{9hjd2}) we have
\begin{equation}\label{long}
E|\sigma(\theta,X_{\tau}(\theta),X_{\tau}(\theta-\tau))|^{p}\leq C_{p,L,T},\,\,\forall\,\,0\leq \theta\leq T.
\end{equation}
It follows from equations (\ref{mall1}) and (\ref{mall2}) that the Malliavin derivative $D_\theta X_{\tau}(t)$ satisfies
\begin{align}
		D_\theta X_{\tau}(t)&=\sigma(\theta,X_{\tau}(\theta),X_{\tau}(\theta-\tau))
		+\int_{\theta}^{t} \bar{b}^{\tau}_{2}(s)D_\theta X_{\tau}(s)ds+\int_{\theta}^{t} \bar{b}^{\tau}_{3}(s)D_\theta X_{\tau}(s-\tau)\ind_{[\theta+\tau,t]}(s)ds\nonumber\\
		&+\int_{\theta}^{t}\bar{\sigma}^{\tau}_{2}(s)D_\theta X_{\tau}(s)dB(s)+\int_{\theta}^{t}\bar{\sigma}^{\tau}_{3}(s)D_\theta X_{\tau}(s-\tau)\ind_{[\theta+\tau,t]}(s)dB(s),\,\,\,0\leq \theta\leq t\leq T.\label{jkf78}
	\end{align}
We therefore get
\begin{align*}
			E&|D_\theta X_{\tau}(t)|^{p}\\
&\leq 5^{p-1}\bigg( E|\sigma(\theta,X_{\tau}(\theta),X_{\tau}(\theta-\tau))|^{p}
			+E\left|\int_{\theta}^{t} \bar{b}^{\tau}_{2}(s)D_\theta X_{\tau}(s)ds\right|^{p}+E\left|\int_{\theta}^{t} \bar{b}^{\tau}_{3}(s)D_\theta X_{\tau}(s-\tau)\ind_{[\theta+\tau,t]}(s)ds\right|^{p}\nonumber\\
			&\quad+E\left|\int_{\theta}^{t}\bar{\sigma}^{\tau}_{2}(s)D_\theta X_{\tau}(s)dB(s)\right|^{p}+E\left|\int_{\theta}^{t}\bar{\sigma}^{\tau}_{3}(s)D_\theta X_{\tau}(s-\tau)\ind_{[\theta+\tau,t]}(s)dB(s)\right|^{p}\bigg),\,\,\,0\leq \theta\leq t\leq T.
\end{align*}
By the boundedness of $\bar{b}^{\tau}_{2}(s)$, $\bar{b}^{\tau}_{3}(s)$, $\bar{\sigma}^{\tau}_{2}(s)$, $\bar{\sigma}^{\tau}_{3}(s),$ the estimate (\ref{long}) and the H\"{o}lder and Burkholder-Davis-Gundy inequalities, it is easy to see that
\begin{align*}
E|D_\theta X_{\tau}(t)|^{p}&\leq C_{p,L,T} + C_{p,L,T}\int_{\theta}^{t}E|D_{\theta}X_{\tau}(s)|^{p}ds+ C_{p,L,T}\int_{\theta}^{t}E|D_{\theta}X_{\tau}(s-\tau)|^{p}ds\\
&\leq C_{p,L,T} + C_{p,L,T}\int_{\theta}^{t}E|D_{\theta}X_{\tau}(s)|^{p}ds,\,\,\,0\leq \theta\leq t\leq T,
\end{align*}
where $C_{p,L,T}$ is some positive constant. So, we can obtain (\ref{lem2.1}) by using Gronwall's lemma.

In order to prove (\ref{lem2.2}), we assume, without the loss of generally, that $s < t$. We consider three cases separately.

\noindent{\em Case 1: $0\leq \theta \leq s-\tau$}. From the equation \eqref{mall1}, we have
\begin{align*}
D_{\theta}X_{\tau}(t)-D_{\theta}X_{\tau}(s) &= \int_{s}^{t}\bar{b}_{2}^{\tau}(u)D_{\theta}X_{\tau}(u)du+\int_{s}^{t}\bar{b}_{3}^{\tau}(u)D_{\theta}X_{\tau}(u-\tau)du\\
		&+\int_{s}^{t}\bar{\sigma}_{2}^{\tau}(u)D_{\theta}X_{\tau}(u)dB_{u}+\int_{s}^{t}\bar{\sigma}_{3}^{\tau}(u)D_{\theta}X_{\tau}(u-\tau)du.
	\end{align*}
	Using  the H\"{o}lder and Burkholder-Davis-Gundy inequalities and (\ref{lem2.1}), we have
	\begin{align*}
		E|D_{\theta}X_{\tau}(t)-D_{\theta}X_{\tau}(s)|^{p} &\leq C_{p,L,T} \left((t-s)^{p-1}\int_{s}^{t}E|D_{\theta}X_{\tau}(u)|^{p}du + (t-s)^{\frac{p}{2}-1} \int_{s}^{t}E|D_{\theta}X_{\tau}(u)|^{p}du\right)\\
		&\leq C_{p,L,T}|t-s|^{p/2}.
	\end{align*}
\noindent{\em Case 2: $s-\tau < \theta \leq t-\tau$}. From the equations \eqref{mall1} and \eqref{mall2} we have
		\begin{align*}
	D_{\theta}X_{\tau}(t)-D_{\theta}X_{\tau}(s) &= \int_{s}^{t}\bar{b}_{2}^{\tau}(u)D_{\theta}X_{\tau}(u)du+\int_{\theta+\tau}^{t}\bar{b}_{3}^{\tau}(u)D_{\theta}X_{\tau}(u-\tau)du\\ &+\int_{s}^{t}\bar{\sigma}_{2}^{\tau}(u)D_{\theta}X_{\tau}(u)dB_{u}+\int_{\theta+\tau}^{t}\bar{\sigma}_{3}^{\tau}(u)D_{\theta}X_{\tau}(u-\tau)du.
	\end{align*}
	Using the same arguments as in the proof of {\it Case 1}, we obtain
	\begin{align*}
	E|D_{\theta}X_{\tau}(t)-D_{\theta}X_{\tau}(s)|^{p} &\leq C_{p,L,T} \left(|t-s|^{p/2}+|t-\theta-\tau|^{p/2}\right)\\
	&\leq C_{p,L,T}|t-s|^{p/2}.
	\end{align*}
\noindent{\em Case 3: $t-\tau < \theta \leq t$}. From  $\eqref{mall2}$ we have
	\begin{align*}
	D_{\theta}X_{\tau}(t)-D_{\theta}X_{\tau}(s)= \int_{s}^{t}\bar{b}_{2}^{\tau}(u)D_{\theta}X_{\tau}(u)du+
	\int_{s}^{t}\bar{\sigma}_{2}^{\tau}(u)D_{\theta}X_{\tau}(u)dB_{u},
	\end{align*}
and hence, we also have
$$E|D_{\theta}X_{\tau}(t)-D_{\theta}X_{\tau}(s)|^{p}\leq C_{p,L,T}|t-s|^{p/2}.$$
This finishes the proof of Proposition.	
	\end{proof}

\begin{prop}\label{lem4}
Let Assumptions \ref{assum1} hold and, in addition, we assume that
$$|\sigma(t,x,y)|\geq \sigma_0>0,\,\,\forall\,t\in [0,T],x,y\in \mathbb{R}.$$
Then, for every $p\geq 1$ and for all  $0<t\leq T,$ we have
\begin{equation}\label{lem4.1}
E\left[\frac{1}{\|DX_{\tau}(t)\|^{2p}_{L^2[0,T]}}\right]\leq
C_{p,L,T}\,t^{-p},
\end{equation}
where $C_{p,L,T}$ is a positive constant.
%where $C_{p,L,T}$ is a positive constant not depending on $t.$
\end{prop}
\begin{proof} Fixed $t\in (0,T].$ By using the fundamental inequality $(a+b+c)^{2}\geq \frac{a^{2}}{2}- 2(b^{2}+c^{2}),$ we obtain from the equation (\ref{jkf78}) that
\begin{align}
		|D_\theta X_{\tau}(t)|^2&\geq \frac{1}{2}\sigma^2(\theta,X_{\tau}(\theta),X_{\tau}(\theta-\tau))
		-2\bigg(\int_{\theta}^{t} \bar{b}^{\tau}_{2}(s)D_\theta X_{\tau}(s)ds+\int_{\theta}^{t} \bar{b}^{\tau}_{3}(s)D_\theta X_{\tau}(s-\tau)\ind_{[\theta+\tau,t]}(s)ds\bigg)^2\nonumber\\
		&-2\bigg(\int_{\theta}^{t}\bar{\sigma}^{\tau}_{2}(s)D_\theta X_{\tau}(s)dB(s)+\int_{\theta}^{t}\bar{\sigma}^{\tau}_{3}(s)D_\theta X_{\tau}(s-\tau)\ind_{[\theta+\tau,t]}(s)dB(s)\bigg)^2,\,\,\,0\leq \theta\leq t\leq T.\notag
	\end{align}
For each $y\geq y_0:=\frac{4}{t\sigma_0^2},$ the real number $\varepsilon:=\frac{4}{y\sigma_0^2t}$ belongs to $(0,1].$  Hence,
\begin{align}%\label{3.12}
 \|DX_{\tau}(t)\|^{2}_{L^2[0,T]}& \geq\int_{t(1-\varepsilon)}^{t}|D_\theta X_{\tau}(t)|^2d\theta\geq\int_{t(1-\varepsilon)}^{t}\frac{\sigma^2(\theta,X_{\tau}(\theta),X_{\tau}(\theta-\tau))}{2}d\theta\notag\\
 &- 2\int_{t(1-\varepsilon)}^{t}\bigg(\int_{\theta}^{t} \bar{b}^{\tau}_{2}(s)D_\theta X_{\tau}(s)ds+\int_{\theta}^{t} \bar{b}^{\tau}_{3}(s)D_\theta X_{\tau}(s-\tau)\ind_{[\theta+\tau,t]}(s)ds\bigg)^2d\theta \nonumber \\
& - 2\int_{t(1-\varepsilon)}^{t}\bigg(\int_{\theta}^{t}\bar{\sigma}^{\tau}_{2}(s)D_\theta X_{\tau}(s)dB(s)+\int_{\theta}^{t}\bar{\sigma}^{\tau}_{3}(s)D_\theta X_{\tau}(s-\tau)\ind_{[\theta+\tau,t]}(s)dB(s)\bigg)^2d\theta\nonumber\\
&\geq\frac{\sigma_0^2t\varepsilon}{2} - I_{y}(t)=\frac{2}{y} - I_{y}(t),\nonumber
\end{align}
where
\begin{align*}
I_{y}(t)&:=2\int_{t(1-\varepsilon)}^{t}\bigg(\int_{\theta}^{t} \bar{b}^{\tau}_{2}(s)D_\theta X_{\tau}(s)ds+\int_{\theta}^{t} \bar{b}^{\tau}_{3}(s)D_\theta X_{\tau}(s-\tau)\ind_{[\theta+\tau,t]}(s)ds\bigg)^2d\theta \\
&+2\int_{t(1-\varepsilon)}^{t} \bigg(\int_{\theta}^{t}\bar{\sigma}^{\tau}_{2}(s)D_\theta X_{\tau}(s)dB(s)+\int_{\theta}^{t}\bar{\sigma}^{\tau}_{3}(s)D_\theta X_{\tau}(s-\tau)\ind_{[\theta+\tau,t]}(s)dB(s)\bigg)^2d\theta.
\end{align*}
Then, by Markov inequality, we obtain
\begin{align}\label{3.13}
P\left(\|DX_{\tau}(t)\|^{2}_{L^2[0,T]}\leq \frac{1}{y}\right) \leq P\left(\frac{2}{y} - I_{y}(t)\leq \frac{1}{y}\right)= P\left(I_{y}(t)\geq \frac{1}{y}\right)\leq y^{q/2}E\left(|I_{y}(t)|^{q/2}\right)\,\,\,\forall\,\,q\geq 2.
\end{align}
By the inequality $(|a|+|b|)^{q/2}\leq 2^{q/2-1}(|a|^{q/2}+(|b|^{q/2}),$ we get
\begin{align}
&E|I_{y}(t)|^{q/2}\nonumber \\
&\leq 2^{q-1}\bigg[E\left(\bigg(\int_{\theta}^{t} \bar{b}^{\tau}_{2}(s)D_\theta X_{\tau}(s)ds+\int_{\theta}^{t} \bar{b}^{\tau}_{3}(s)D_\theta X_{\tau}(s-\tau)\ind_{[\theta+\tau,t]}(s)ds\bigg)^2d\theta\right)^{q/2} \notag\\
&+ E\left(\int_{t(1-\varepsilon)}^{t} \bigg(\int_{\theta}^{t}\bar{\sigma}^{\tau}_{2}(s)D_\theta X_{\tau}(s)dB(s)+\int_{\theta}^{t}\bar{\sigma}^{\tau}_{3}(s)D_\theta X_{\tau}(s-\tau)\ind_{[\theta+\tau,t]}(s)dB(s)\bigg)^2d\theta\right)^{q/2}\bigg].\nonumber
\end{align}
By using the H\"{o}lder and Burkholder-Davis-Gundy inequalities, it follows from (\ref{lem2.1}) that
\begin{align}
&E|I_{y}(t)|^{q/2}\nonumber \\
%&\leq C(t\varepsilon)^{\frac{p-2}{2}}\left(\int_{t(1-\varepsilon)}^{t}E\left(\int_{\theta}^{t}|D_{\theta}X_{s}|^2ds\right)^{p/2}d\theta + \int_{t(1-\varepsilon)}^{t} E\left(\int_{\theta}^{t}|D_{\theta}X_{s}|dB_{s}\right)^{p}d\theta\right) \nonumber\\
&\leq C_{q,L,T}(t\varepsilon)^{\frac{q-2}{2}}\left(\int_{t(1-\varepsilon)}^{t}(t-\theta)^{\frac{q-2}{2}}\int_{\theta}^{t}E|D_{\theta}X_{\tau}(s)|^{p}dsd\theta + \int_{t(1-\varepsilon)}^{t} \left(\int_{\theta}^{t}E|D_{\theta}X_{\tau}(s)|^{2}ds\right)^{q/2}d\theta\right) \nonumber\\
&\leq C_{q,L,T}(t\varepsilon)^{\frac{q-2}{2}}\left(\int_{t(1-\varepsilon)}^{t}(t-\theta)^{\frac{q}{2}}d\theta + \int_{t(1-\varepsilon)}^{t} \left(t-\theta\right)^{q/2}d\theta\right) \nonumber\\
&\leq C_{q,L,T}(t\varepsilon)^{\frac{q-2}{2}} (t\varepsilon)^{\frac{q}{2}+1}\nonumber\\
&= C_{q,L,T}\left(\frac{4}{y\sigma_0^2}\right)^{q},\label{3.14}
\end{align}
where $C_{q,L,T}$ is a positive constant. Combining $\eqref{3.13}$ and $\eqref{3.14}$ we deduce
\begin{align*}
P\left(\|DX_{\tau}(t)\|^{2}_{L^2[0,T]}\leq \frac{1}{y}\right)\leq C_{q,L,T}y^{p/2}\left(\frac{4}{y\sigma_0^2}\right)^{p}\,\,\,\forall\,\,p\geq 2,y\geq y_0.
\end{align*}
For any $p\geq 2$ and $q=2p+1$, we obtain the following estimates
\begin{align}
E\left(\|DX_{\tau}(t)\|^{-2p}_{L^2[0,T]}\right) & = \int_{0}^{\infty}p y^{p -1}P\left(\|DX_{\tau}(t)\|^{-2}_{L^2[0,T]}>y\right)dy\nonumber\\
  &\leq  \int_{0}^{y_0}p y^{p -1}dy + \int_{y_0}^{\infty}p y^{p -1}P\left(\|DX_{\tau}(t)\|^{2}_{L^2[0,T]}<\frac{1}{y}\right)dy\nonumber\\
&\leq y_0^{p}+ pC_{p,L,T}\int_{y_0}^{\infty}y^{p -1}y^{q/2}\left(\frac{4}{y\sigma_0^2}\right)^{q}dy\nonumber\\
& =y_0^{p} +2p C_{p,L,T}\left(\frac{4}{\sigma_0^2}\right)^{2p+1}y_0^{-\frac{1}{2}}.\label{3.15}
\end{align}
We  recall here that $y_0=\frac{4}{t\sigma_0^2}.$ So (\ref{lem4.1}) follows from (\ref{3.15}).
\end{proof}

\begin{prop}\label{9jk3k}Let Assumptions \ref{assum1} and \ref{assum5} hold. It holds that
$$E|D_{r}D_{\theta}X_{\tau}(t)|^{4} \leq C_{T,L}\,\,\forall\,0\leq \theta,r\leq t\leq T,$$ %for all $0\leq \theta,r\leq t\leq T.$
where $C_{L,T}$ is a positive constant.
\end{prop}
\begin{proof} We rewrite the equations (\ref{rem1.1}) and (\ref{rem1.2}) as follows
\begin{align}
D_\theta& X_{\tau}(t)=\sigma(\theta,X_{\tau}(\theta),X_{\tau}(\theta-\tau))
+\int_\theta^t b_2'(s,X_{\tau}(s),X_{\tau}(s-\tau))D_\theta X_{\tau}(s)ds\notag\\
&+\int_{\theta}^t b_3'(s,X_{\tau}(s),X_{\tau}(s-\tau))D_\theta X_{\tau}(s-\tau)\ind_{[\theta+\tau,t]}(s)ds+\int_\theta^t\sigma_2'(s,X_{\tau}(s),X_{\tau}(s-\tau))D_\theta X_{\tau}(s)dB(s)\notag\\
&+\int_{\theta}^t\sigma_3'(s,X_{\tau}(s),X_{\tau}(s-\tau))D_\theta X_{\tau}(s-\tau)\ind_{[\theta+\tau,t]}(s)dB(s),\,\,0\leq \theta\leq t\leq T.\notag
\end{align}
Hence, the second order Malliavin derivative of $X_{\tau}(t)$ can be computed by
\begin{align}
&D_rD_\theta X_{\tau}(t)=D_r[\sigma(\theta,X_{\tau}(\theta),X_{\tau}(\theta-\tau))]+\sigma_2'(r,X_{\tau}(r),X_{\tau}(r-\tau))D_\theta X_{\tau}(r)\notag\\
&+\sigma_3'(r,X_{\tau}(r),X_{\tau}(r-\tau))D_\theta X_{\tau}(r-\tau)\ind_{[\theta+\tau,t]}(r)+\int_{\theta \vee r}^t D_r[b_2'(s,X_{\tau}(s),X_{\tau}(s-\tau))]D_\theta X_{\tau}(s)ds\notag\\
&+\int_{\theta \vee r}^t b_2'(s,X_{\tau}(s),X_{\tau}(s-\tau))D_rD_\theta X_{\tau}(s)ds+\int_{\theta \vee r}^t D_r[b_3'(s,X_{\tau}(s),X_{\tau}(s-\tau))]D_\theta X_{\tau}(s-\tau)\ind_{[\theta+\tau,t]}(s)ds\notag\\
&+\int_{\theta \vee r}^t b_3'(s,X_{\tau}(s),X_{\tau}(s-\tau))D_rD_\theta X_{\tau}(s-\tau)\ind_{[\theta+\tau,t]}(s)ds+\int_{\theta \vee r}^tD_r[\sigma_2'(s,X_{\tau}(s),X_{\tau}(s-\tau))]D_\theta X_{\tau}(s)dB(s)\notag\\
&+\int_{\theta \vee r}^t\sigma_2'(s,X_{\tau}(s),X_{\tau}(s-\tau))D_rD_\theta X_{\tau}(s)dB(s)+\int_{\theta \vee r}^tD_r[\sigma_3'(s,X_{\tau}(s),X_{\tau}(s-\tau))]D_\theta X_{\tau}(s-\tau)\ind_{[\theta+\tau,t]}(s)dB(s)\notag\\
&+\int_{\theta \vee r}^t\sigma_3'(s,X_{\tau}(s),X_{\tau}(s-\tau))D_rD_\theta X_{\tau}(s-\tau)\ind_{[\theta+\tau,t]}(s)dB(s),\,\,0\leq r,\theta\leq t\leq T,\notag
\end{align}
where, for $h\in \{b_2',b_3',\sigma_2',\sigma_3'\},$ we have
$$D_{r}\left[h(s,X_{\tau}(s),X_{\tau}(s-\tau))\right] = h_2'(s,X_{\tau}(s),X_{\tau}(s-\tau))D_r X_{\tau}(s)+ h_3'(s,X_{\tau}(s),X_{\tau}(s-\tau))D_{r} X_{\tau}(s-\tau).$$
Note that the partial derivatives of $b$ and $\sigma$ are bounded. By using (\ref{lem2.1}) and  the H\"{o}lder and Burkholder-Davis-Gundy inequalities, we verify that% and (\ref{lem2.1})
\begin{align*}
E|D_{r}D_{\theta}X_{\tau}(t)|^{4}&\leq C_{T,L}+ C_{T,L}\int_{\theta \vee r}^{t}E|D_{r}D_{\theta}X_{\tau}(s)|^{4}ds+ C_{T,L}\int_{\theta \vee r}^{t}E|D_{r}D_{\theta}X_{\tau}(s-\tau)|^{4}ds\\
&\leq C_{T,L}+ C_{T,L}\int_{\theta \vee r}^{t}E|D_{r}D_{\theta}X_{\tau}(s)|^{4}ds,\,\,0\leq r,\theta\leq t\leq T,
\end{align*}
which, by Gronwall's lemma, gives us the desired conclusion.
% $$E|D_{\theta}D_{r}X_{\tau_{1}}(t)|^{4} \leq C_{T,L}$$ for all $0\leq \theta,r\leq t\leq T.$
\end{proof}

\subsection{$L^p$-distances and weak convergence}
In this Subsection, we consider the equations (\ref{cara04}) and (\ref{cara05}). We first bound $L^p$-distances between the solutions and between their Malliavin derivatives. Then, we use Lemma \ref{lm2.1} to estimate the weak rate of  convergence. In the whole subsection, we assume without loss of generality that $\tau_1<\tau_2.$
\begin{lem}\label{ujl9s}Let Assumptions \ref{assum1} hold and let $h:[0,T]\times \mathbb{R}^2\to \mathbb{R}$ be such that
$$|h(t,x_1,y_1)-h(t,x_2,y_2)|\leq L(|x_1-x_2|+|y_1-y_2|)$$
for all $t\in [0,T],x_1,x_2,y_1,y_2\in \mathbb{R}.$ Then, for every $p\geq 2$ and for all $t\in[0,T],$
\begin{align}
&\int_{0}^{t}E|h(s,X_{\tau_1}(s),X_{\tau_1}(s-\tau_1))-h(s,X_{\tau_2}(s),X_{\tau_2}(s-\tau_2))|^pds\notag\\
&\leq  C_{p,L,T}\bigg(t|\tau_{1}-\tau_{2}|^{p/2}+\int_{0}^{t\wedge \tau_1} |\varphi(s-\tau_1)-\varphi(s-\tau_2)|^{p}ds
+\int_{t\wedge \tau_1}^{t\wedge \tau_2} |\varphi(0)-\varphi(s-\tau_2)|^{p}ds\nonumber\notag\\
	&\hspace{8.8cm}+\int_{0}^{t}E|X_{\tau_{1}}(s)-X_{\tau_{2}}(s)|^{p}ds\bigg),\label{mpt2a}
\end{align}
where $C_{p,L,T}$ is a positive constant.
\end{lem}
\begin{proof}We have
\begin{align}
\int_{0}^{t}E|h(s,&X_{\tau_1}(s),X_{\tau_1}(s-\tau_1))-h(s,X_{\tau_2}(s),X_{\tau_2}(s-\tau_2))|^pds\notag\\
&\leq L^{p}\int_{0}^{t}E(|X_{\tau_{1}}(s)-X_{\tau_{2}}(s)| +|X_{\tau_{1}}(s-\tau_{1})-X_{\tau_{2}}(s-\tau_{2})|)^{p}ds\notag\\
&\leq L^{p}2^{p-1}\int_{0}^{t}(E|X_{\tau_{1}}(s)-X_{\tau_{2}}(s)|^{p} +E|X_{\tau_{1}}(s-\tau_{1})-X_{\tau_{2}}(s-\tau_{2})|^{p})ds\notag\\
&=L^{p}2^{p-1}M_p(t)+L^{p}2^{p-1}\int_{0}^{t}E|X_{\tau_{1}}(s)-X_{\tau_{2}}(s)|^{p}ds,\,\,\,t\in[0,T],\label{mpt2n}
\end{align}
where
\begin{align*}%\label{mpt1}
	M_p(t)&:=\int_{0}^{t}E|X_{\tau_{1}}(s-\tau_{1})-X_{\tau_{2}}(s-\tau_{2})|^{p}ds\nonumber\\
	&\leq 2^{p-1}\int_{0}^{t} (E|X_{\tau_{1}}(s-\tau_{1})-X_{\tau_{2}}(s-\tau_{1})|^{p}+E|X_{\tau_{2}}(s-\tau_{1})-X_{\tau_{2}}(s-\tau_{2})|^{p})ds.
\end{align*}
Since $X_{\tau_1}(s)=X_{\tau_2}(s)=\varphi(s),s\in[-\tau_1,0],$ it holds that
\begin{align*}
\int_{0}^{t}&E|X_{\tau_{1}}(s-\tau_{1})-X_{\tau_{2}}(s-\tau_{1})|^{p}ds\\
&=\int_{0}^{t\wedge \tau_1}E|X_{\tau_{1}}(s-\tau_{1})-X_{\tau_{2}}(s-\tau_{1})|^{p}ds+\int_{t\wedge \tau_1}^t E|X_{\tau_{1}}(s-\tau_{1})-X_{\tau_{2}}(s-\tau_{1})|^{p}ds\\
&=\int_{t\wedge \tau_1}^t E|X_{\tau_{1}}(s-\tau_{1})-X_{\tau_{2}}(s-\tau_{1})|^{p}ds=\int_{0}^{t-\tau_1} E|X_{\tau_{1}}(s)-X_{\tau_{2}}(s)|^{p}ds\\
&\leq \int_{0}^{t}E|X_{\tau_{1}}(s)-X_{\tau_{2}}(s)|^{p}ds,\,\,\,t\in[0,T].
\end{align*}
On the other hand, we have
\begin{align*}
&\int_{0}^{t} E|X_{\tau_2}(s-\tau_1)-X_{\tau_2}(s-\tau_2)|^{p}ds= \int_{0}^{t\wedge \tau_1} E|X_{\tau_2}(s-\tau_1)-X_{\tau_2}(s-\tau_2)|^{p}ds\\
&+\int_{t\wedge \tau_1}^{t\wedge \tau_2} E|X_{\tau_2}(s-\tau_1)-X_{\tau_2}(s-\tau_2)|^{p}ds+\int_{t\wedge \tau_2}^{t} E|X_{\tau_2}(s-\tau_1)-X_{\tau_2}(s-\tau_2)|^{p}ds,\,\,\,t\in[0,T].
\end{align*}
Then, recalling (\ref{9hjd1}), we obtain
\begin{align*}
&\int_{0}^{t} E|X_{\tau_2}(s-\tau_1)-X_{\tau_2}(s-\tau_2)|^{p}ds\\
&\leq \int_{0}^{t\wedge \tau_1} |\varphi(s-\tau_1)-\varphi(s-\tau_2)|^{p}ds+2^{p-1}\int_{t\wedge \tau_1}^{t\wedge \tau_2} E|X_{\tau_2}(s-\tau_1)-\varphi(0)|^{p}ds\\
&+2^{p-1}\int_{t\wedge \tau_1}^{t\wedge \tau_2} |\varphi(0)-X_{\tau_2}(s-\tau_2)|^{p}ds+\int_{t\wedge \tau_2}^{t} E|X_{\tau_2}(s-\tau_1)-X_{\tau_2}(s-\tau_2)|^{p}ds\\
&\leq \int_{0}^{t\wedge \tau_1} |\varphi(s-\tau_1)-\varphi(s-\tau_2)|^{p}ds+C_{p,L,T}\int_{t\wedge \tau_1}^{t\wedge \tau_2} (s-\tau_1)^{p/2}ds\\
&+2^{p-1}\int_{t\wedge \tau_1}^{t\wedge \tau_2} |\varphi(0)-\varphi(s-\tau_2)|^{p}ds+C_{p,L,T}\int_{t\wedge \tau_2}^{t} (\tau_{2}-\tau_{1})^{p/2}ds\\
&\leq C_{p,L,T}\bigg(t|\tau_{1}-\tau_{2}|^{p/2}+\int_{0}^{t\wedge \tau_1} |\varphi(s-\tau_1)-\varphi(s-\tau_2)|^{p}ds
+\int_{t\wedge \tau_1}^{t\wedge \tau_2} |\varphi(0)-\varphi(s-\tau_2)|^{p}ds\bigg)
\end{align*}
Hence,
\begin{align}
	M_p(t)&\leq  C_{p,L,T}\bigg(t|\tau_{1}-\tau_{2}|^{p/2}+\int_{0}^{t\wedge \tau_1} |\varphi(s-\tau_1)-\varphi(s-\tau_2)|^{p}ds
+\int_{t\wedge \tau_1}^{t\wedge \tau_2} |\varphi(0)-\varphi(s-\tau_2)|^{p}ds\nonumber\\
	&\hspace{6cm}+\int_{0}^{t}E|X_{\tau_{1}}(s)-X_{\tau_{2}}(s)|^{p}ds\bigg),\,\,\,t\in[0,T].\label{mpt2}
\end{align}
As a consequence, we obtain (\ref{mpt2a}) by inserting (\ref{mpt2}) into (\ref{mpt2n}). The proof of the lemma is complete.
\end{proof}
\begin{prop}\label{them1} Let Assumption \ref{assum1} hold. Then, for every $p\geq 2$ and for all $t\in[0,T],$ there exists a positive constants $C_{p,L,T}$ such that
\begin{align}
E|X_{\tau_1}(t)-X_{\tau_2}(t)|^p&\leq  C_{p,L,T}t^{\frac{p}{2}-1}\bigg(t|\tau_{1}-\tau_{2}|^{p/2}\notag\\
&+\int_{0}^{t\wedge \tau_1} |\varphi(s-\tau_1)-\varphi(s-\tau_2)|^{p}ds
+\int_{t\wedge \tau_1}^{t\wedge \tau_2} |\varphi(0)-\varphi(s-\tau_2)|^{p}ds\bigg).\label{3.them1}
\end{align}
\end{prop}
\begin{proof} We write
\begin{align}\label{prthem3.1}
X_{\tau_1}(t)-X_{\tau_2}(t)=I_1(t)+I_2(t),\,\,\,0\leq t\leq T,
\end{align}
where the terms $I_1(t),I_2(t)$ are defined by
\begin{align*}
&I_1(t):=\int_0^t [b(s,X_{\tau_1}(s),X_{\tau_1}(s-\tau_1))-b(s,X_{\tau_2}(s),X_{\tau_2}(s-\tau_2))]ds,\\
&I_2(t):=\int_0^t[\sigma(s,X_{\tau_1}(s),X_{\tau_1}(s-\tau_1))-\sigma(s,X_{\tau_2}(s),X_{\tau_2}(s-\tau_2))]dB(s).
\end{align*}
Using the H\"{o}lder inequality and Lemma \ref{ujl9s}, we deduce
\begin{align}
E|I_1(t)|^p&\leq t^{p-1}\int_{0}^{t}E|b(s,X_{\tau_1}(s),X_{\tau_1}(s-\tau_1))-b(s,X_{\tau_2}(s),X_{\tau_2}(s-\tau_2))|^pds\notag\\
&\leq C_{p,L,T}t^{p-1}\left(t|\tau_{1}-\tau_{2}|^{p/2}+\int_{0}^{t\wedge \tau_1} |\varphi(s-\tau_1)-\varphi(s-\tau_2)|^{p}ds
+\int_{t\wedge \tau_1}^{t\wedge \tau_2} |\varphi(0)-\varphi(s-\tau_2)|^{p}ds\right)\nonumber\\
& \ \ +C_{p,L,T}\int_{0}^{t}E|X_{\tau_{1}}(s)-X_{\tau_{2}}(s)|^{p}ds,\,\,\,0\leq t\leq T.\label{3.03}
\end{align}
Similarly, we also have
	\begin{align}\label{3.04}
	E|I_2(t)|^p&\leq C_{p,L,T}  E\left(\int_{0}^{t}|\sigma(s,X_{\tau_1}(s),X_{\tau_1}(s-\tau_1))-\sigma(s,X_{\tau_2}(s),X_{\tau_2}(s-\tau_2))|^{2}ds\right)^{p/2}\nonumber\\
	&\leq  C_{p,L,T} t^{\frac{p}{2}-1}\int_{0}^{t}E|\sigma(s,X_{\tau_1}(s),X_{\tau_1}(s-\tau_1))-\sigma(s,X_{\tau_2}(s),X_{\tau_2}(s-\tau_2))|^{p}ds\nonumber\\
%	&\leq (2t)^{\frac{p}{2}-1}L^{p}\int_{0}^{t}(E|X_{\tau_{1}}(s)-X_{\tau_{1}}(s)|^{p} +E|X_{\tau_{1}}(s-\tau_{1})-X_{\tau_{2}}(s-\tau_{2})|^{p})ds\nonumber\\
	&\leq C_{p,L,T}t^{\frac{p}{2}-1}\left(t|\tau_{1}-\tau_{2}|^{p/2}+\int_{0}^{t\wedge \tau_1} |\varphi(s-\tau_1)-\varphi(s-\tau_2)|^{p}ds
+\int_{t\wedge \tau_1}^{t\wedge \tau_2} |\varphi(0)-\varphi(s-\tau_2)|^{p}ds\right) \nonumber\\
	& \ \ +C_{p,L,T}\int_{0}^{t}E|X_{\tau_{1}}(s)-X_{\tau_{2}}(s)|^{p}ds,\,\,\,0\leq t\leq T.
	\end{align}
Combining \eqref{prthem3.1}, \eqref{3.03} and \eqref{3.04}, we obtain
\begin{align*}
E|X_{\tau_1}(t)&-X_{\tau_2}(t)|^p\\
&\leq C_{p,L,T}t^{\frac{p}{2}-1}\left(t|\tau_{1}-\tau_{2}|^{p/2}+\int_{0}^{t\wedge \tau_1} |\varphi(s-\tau_1)-\varphi(s-\tau_2)|^{p}ds
+\int_{t\wedge \tau_1}^{t\wedge \tau_2} |\varphi(0)-\varphi(s-\tau_2)|^{p}ds\right) \nonumber\\
&\ \ +C_{p,L,T}\int_{0}^{t}E|X_{\tau_{1}}(s)-X_{\tau_{2}}(s)|^{p}ds,\,\,\,0\leq t\leq T.
\end{align*}
Since the function $t\mapsto t|\tau_{1}-\tau_{2}|^{p/2}+\int_{0}^{t\wedge \tau_1} |\varphi(s-\tau_1)-\varphi(s-\tau_2)|^{p}ds
+\int_{t\wedge \tau_1}^{t\wedge \tau_2} |\varphi(0)-\varphi(s-\tau_2)|^{p}ds$ is non-decreasing, we can use the Gronwall-type lemma (see, Theorem 1.4.2 in \cite{Pachpatte1998}) to get
\begin{align*}
&E|X_{\tau_1}(t)-X_{\tau_2}(t)|^p\\
&\leq C_{p,L,T}t^{\frac{p}{2}-1}\left(t|\tau_{1}-\tau_{2}|^{p/2}+\int_{0}^{t\wedge \tau_1} |\varphi(s-\tau_1)-\varphi(s-\tau_2)|^{p}ds
+\int_{t\wedge \tau_1}^{t\wedge \tau_2} |\varphi(0)-\varphi(s-\tau_2)|^{p}ds\right)e^{C_{p,L,T} t}\\
&\leq C_{p,L,T}t^{\frac{p}{2}-1}\left(t|\tau_{1}-\tau_{2}|^{p/2}+\int_{0}^{t\wedge \tau_1} |\varphi(s-\tau_1)-\varphi(s-\tau_2)|^{p}ds
+\int_{t\wedge \tau_1}^{t\wedge \tau_2} |\varphi(0)-\varphi(s-\tau_2)|^{p}ds\right).
\end{align*}
The proof of Proposition is complete.
\end{proof}
\begin{cor}\label{imld5k} Under the assumptions of Lemma \ref{ujl9s}, we have
\begin{align*}
&\int_{0}^{t}E|h(s,X_{\tau_1}(s),X_{\tau_1}(s-\tau_1))-h(s,X_{\tau_2}(s),X_{\tau_2}(s-\tau_2))|^pds\\
&\leq  C_{p,L,T}\bigg(t|\tau_{1}-\tau_{2}|^{p/2}+\int_{0}^{t\wedge \tau_1} |\varphi(s-\tau_1)-\varphi(s-\tau_2)|^{p}ds
+\int_{t\wedge \tau_1}^{t\wedge \tau_2} |\varphi(0)-\varphi(s-\tau_2)|^{p}ds\bigg)\\
&\leq  C_{p,L,T}\bigg(|\tau_{1}-\tau_{2}|^{p/2}+\int_{0}^{t\wedge \tau_1} |\varphi(s-\tau_1)-\varphi(s-\tau_2)|^{p}ds
+\int_{t\wedge \tau_1}^{t\wedge \tau_2} |\varphi(0)-\varphi(s-\tau_2)|^{p}ds\bigg)
\end{align*}
for every $p\geq 2$ and for all $t\in[0,T],$ where $C_{p,L,T}$ is a positive constant.
\end{cor}
\begin{proof}Follows directly from Lemma \ref{ujl9s} and Proposition \ref{them1}.
\end{proof}
Next, we estimate the distance between the Malliavin derivatives. For this purpose, we set
\begin{align*}
&I_{1,\tau}(\theta,t):=\int_\theta^t b_2'(s,X_{\tau}(s),X_{\tau}(s-\tau))D_\theta X_{\tau}(s)ds,\\
&I_{2,\tau}(\theta,t):=\int_{\theta+\tau}^t b_3'(s,X_{\tau}(s),X_{\tau}(s-\tau))D_\theta X_{\tau}(s-\tau)ds,\\
&J_{1,\tau}(\theta,t):=\int_\theta^t\sigma_2'(s,X_{\tau}(s),X_{\tau}(s-\tau))D_\theta X_{\tau}(s)dB(s),\\
&J_{2,\tau}(\theta,t):=\int_{\theta+\tau}^t\sigma_3'(s,X_{\tau}(s),X_{\tau}(s-\tau))D_\theta X_{\tau}(s-\tau)dB(s),
\end{align*}
%$$I_{1,\tau}(\theta,t):=\int_\theta^t b_2'(s,X_{\tau}(s),X_{\tau}(s-\tau))D_\theta X_{\tau}(s)ds,\,\,\,I_{2,\tau}(\theta,t):=\int_{\theta+\tau}^t b_3'(s,X_{\tau}(s),X_{\tau}(s-\tau))D_\theta X_{\tau}(s-\tau)ds$$
%$$J_{1,\tau}(\theta,t):=\int_\theta^t\sigma_2'(s,X_{\tau}(s),X_{\tau}(s-\tau))D_\theta X_{\tau}(s)dB(s),\,J_{2,\tau}(\theta,t):=\int_{\theta+\tau}^t\sigma_3'(s,X_{\tau}(s),X_{\tau}(s-\tau))D_\theta X_{\tau}(s-\tau)dB(s),$$
where $(X_{\tau}(t))_{t\in[-\tau,T]}$ is the solution to the equation (\ref{eq2}).
\begin{lem}\label{jfks1}Let Assumptions \ref{assum1} and \ref{assum5} hold. We have
\begin{align}
&E|I_{1,\tau_1}(\theta,t)-I_{1,\tau_2}(\theta,t)|^2+E|J_{1,\tau_1}(\theta,t)-J_{1,\tau_2}(\theta,t)|^2\notag\\
&\leq C_{L,T}\left(|\tau_{1}-\tau_{2}|^{2}+\int_{0}^{t\wedge \tau_1} |\varphi(s-\tau_1)-\varphi(s-\tau_2)|^{4}ds
+\int_{t\wedge \tau_1}^{t\wedge \tau_2} |\varphi(0)-\varphi(s-\tau_2)|^{4}ds\right)^{1/2}\nonumber\\
&+C_{L,T}\int_{\theta}^{t} E|D_\theta X_{\tau_1}(s)-D_\theta X_{\tau_2}(s)|^2ds,\,\,0\leq \theta\leq t\leq T,\label{prlem3.1}
\end{align}
where $C_{L,T}$ is a positive constant.
\end{lem}
\begin{proof}
We have
\begin{align*}
I_{1,\tau_1}&(\theta,t)-I_{1,\tau_2}(\theta,t)\\
&=\int_\theta^t [b_2'(s,X_{\tau_1}(s),X_{\tau_1}(s-\tau_1))-b_2'(s,X_{\tau_2}(s),X_{\tau_2}(s-\tau_2))]D_\theta X_{\tau_1}(s)ds\\
&+\int_\theta^t b_2'(s,X_{\tau_2}(s),X_{\tau_2}(s-\tau_2))[D_\theta X_{\tau_1}(s)-D_\theta X_{\tau_2}(s)]ds,\,\,0\leq \theta\leq t.
\end{align*}
Then, by the H\"{o}lder inequality, we obtain
\begin{align}
	 & E|I_{1,\tau_1}(\theta,t)-I_{1,\tau_2}(\theta,t)|^2 \notag                                                                                                                          \\
	 & \leq 2(t-\theta)\int_{\theta}^{t} E|[b_2'(s,X_{\tau_1}(s),X_{\tau_1}(s-\tau_1))-b_2'(s,X_{\tau_2}(s),X_{\tau_2}(s-\tau_2))]D_\theta X_{\tau_1}(s)|^2ds                           \notag\\
	 & +2(t-\theta)\int_\theta^t E|b_2'(s,X_{\tau_2}(s),X_{\tau_2}(s-\tau_2))[D_\theta X_{\tau_1}(s)-D_\theta X_{\tau_2}(s)]|^2ds\notag \\
	 & \leq C_{L,T}\left(\int_{\theta}^{t}E|b_2'(s,X_{\tau_1}(s),X_{\tau_1}(s-\tau_1))-b_2'(s,X_{\tau_2}(s),X_{\tau_2}(s-\tau_2))|^4ds\right)^{1/2}\left(\int_{\theta}^{t}E|D_\theta X_{\tau_1}(s)|^4ds\right)^{1/2} \notag       \\
	 & +C_{L,T}\int_\theta^t E|D_\theta X_{\tau_1}(s)-D_\theta X_{\tau_2}(s)|^2ds,\,\,0\leq \theta\leq t.\label{prlem3.2b}
\end{align}
Note that, by the estimate \eqref{lem2.1}, we have $\int_{\theta}^{t}E|D_\theta X_{\tau_1}(s)|^4ds\leq C_{L,T}.$ Furthermore, an application of Corollary \ref{imld5k} to $h=b_2'$ yields
\begin{align}
&\int_{\theta}^{t}E|b_2'(s,X_{\tau_1}(s),X_{\tau_1}(s-\tau_1))-b_2'(s,X_{\tau_2}(s),X_{\tau_2}(s-\tau_2))|^4ds\notag\\
&\leq C_{L,T}\left(|\tau_{1}-\tau_{2}|^{2}+\int_{0}^{t\wedge \tau_1} |\varphi(s-\tau_1)-\varphi(s-\tau_2)|^{4}ds
+\int_{t\wedge \tau_1}^{t\wedge \tau_2} |\varphi(0)-\varphi(s-\tau_2)|^{4}ds\right),\,\,0\leq \theta\leq t.\notag
\end{align}
Hence, we get
\begin{align}
&E|I_{1,\tau_1}(\theta,t)-I_{1,\tau_2}(\theta,t)|^2\notag\\
&\leq C_{L,T}\left(|\tau_{1}-\tau_{2}|^{2}+\int_{0}^{t\wedge \tau_1} |\varphi(s-\tau_1)-\varphi(s-\tau_2)|^{4}ds
+\int_{t\wedge \tau_1}^{t\wedge \tau_2} |\varphi(0)-\varphi(s-\tau_2)|^{4}ds\right)^{1/2}\nonumber\\
&+C_{L,T}\int_{\theta}^{t} E|D_\theta X_{\tau_1}(s)-D_\theta X_{\tau_2}(s)|^2ds,\,\,0\leq \theta\leq t.\notag
\end{align}
Thus (\ref{prlem3.1}) is verified for $E|I_{1,\tau_1}(\theta,t)-I_{1,\tau_2}(\theta,t)|^2.$  Similarly, we have
\begin{align}
&E|J_{1,\tau_1}(\theta,t)-J_{1,\tau_2}(\theta,t)|^2\leq 2\int_\theta^t E|[\sigma_2'(s,X_{\tau_1}(s),X_{\tau_1}(s-\tau_1))-\sigma_2'(s,X_{\tau_2}(s),X_{\tau_2}(s-\tau_2))]D_\theta X_{\tau_1}(s)|^2ds\notag\\
&+2\int_\theta^t E|\sigma_2'(s,X_{\tau_2}(s),X_{\tau_2}(s-\tau_2))[D_\theta X_{\tau_1}(s)-D_\theta X_{\tau_2}(s)]|^2ds\notag\\
& \leq C\left(\int_{\theta}^{t}E|\sigma_2'(s,X_{\tau_1}(s),X_{\tau_1}(s-\tau_1))-\sigma_2'(s,X_{\tau_2}(s),X_{\tau_2}(s-\tau_2))|^4ds\right)^{1/2}\left(\int_{\theta}^{t}E|D_\theta X_{\tau_1}(s)|^4ds\right)^{1/2}\notag        \\
	 & +C\int_\theta^t E|D_\theta X_{\tau_1}(s)-D_\theta X_{\tau_2}(s)|^2ds,\,\,0\leq \theta\leq t.\label{prlem3.2}
\end{align}
The right hand side of (\ref{prlem3.2}) has the same form as that of (\ref{prlem3.2b}). We therefore can conclude that (\ref{prlem3.1}) also holds true for $E|J_{1,\tau_1}(\theta,t)-J_{1,\tau_2}(\theta,t)|^2.$ The proof of the lemma is complete.
\end{proof}

\begin{lem}\label{jfks1a}Let Assumptions \ref{assum1} and \ref{assum5} hold. Then, we have for $t>\tau_2,$
\begin{align}
&E|I_{2,\tau_1}(\theta,t)-I_{2,\tau_2}(\theta,t)|^2+E|J_{2,\tau_1}(\theta,t)-J_{2,\tau_2}(\theta,t)|^2\notag\\
&\leq C_{L,T}\left(|\tau_{1}-\tau_{2}|^{2}+\int_{0}^{t\wedge \tau_1} |\varphi(s-\tau_1)-\varphi(s-\tau_2)|^{4}ds
+\int_{t\wedge \tau_1}^{t\wedge \tau_2} |\varphi(0)-\varphi(s-\tau_2)|^{4}ds\right)^{1/2}\nonumber\\
&+C_{L,T}\int_{\theta}^{t} E|D_\theta X_{\tau_1}(s)-D_\theta X_{\tau_2}(s)|^2ds,\,\,0\leq \theta\leq t-\tau_2,\label{hjf83}
\end{align}
where $C_{L,T}$ is a positive constant.
\end{lem}
\begin{proof}
When $t>\tau_2$ and $0\leq \theta\leq t-\tau_2$, we have
\begin{align}
&I_{2,\tau_{1}}(\theta,t)-I_{2,\tau_{2}}(\theta,t)\nonumber\\
&=	\int_{\theta+\tau_{1}}^{t}b'_{3}(s,X_{\tau_{1}}(s), X_{\tau_{1}}(s-\tau_{1}))D_{\theta}X_{\tau_{1}}(s-\tau_{1})ds-\int_{\theta+\tau_{2}}^{t}b'_{3}(s,X_{\tau_{2}}(s), X_{\tau_{2}}(s-\tau_{2}))D_{\theta}X_{\tau_{2}}(s-\tau_{2})ds\nonumber\\
&=\int_{\theta+\tau_{2}}^{t}[b'_{3}(s,X_{\tau_{1}}(s), X_{\tau_{1}}(s-\tau_{1}))-b'_{3}(s,X_{\tau_{2}}(s), X_{\tau_{2}}(s-\tau_{2}))]D_{\theta}X_{\tau_{1}}(s-\tau_{1})ds\nonumber\\
&+\int_{\theta+\tau_{1}}^{\theta+\tau_{2}}b'_{3}(s,X_{\tau_{1}}(s), X_{\tau_{1}}(s-\tau_{1}))D_{\theta}X_{\tau_{1}}(s-\tau_{1})ds\notag\\
&+\int_{\theta+\tau_{2}}^{t}b'_{3}(s,X_{\tau_{2}}(s), X_{\tau_{2}}(s-\tau_{2}))[D_{\theta}X_{\tau_{1}}(s-\tau_{1})-D_{\theta}X_{\tau_{2}}(s-\tau_{2})]ds.\label{case1.2}
\end{align}
We observe that the first addend in the right hand side of (\ref{case1.2}) can be estimated as in the proof of Lemma \ref{jfks1} and we obtain
\begin{align*}
E\big|&\int_{\theta+\tau_{2}}^{t}[b'_{3}(s,X_{\tau_{1}}(s), X_{\tau_{1}}(s-\tau_{1}))-b'_{3}(s,X_{\tau_{2}}(s), X_{\tau_{2}}(s-\tau_{2}))]D_{\theta}X_{\tau_{1}}(s-\tau_{1})ds\big|^2\\
&\leq C_{L,T}\left(|\tau_{1}-\tau_{2}|^{2}+\int_{0}^{t\wedge \tau_1} |\varphi(s-\tau_1)-\varphi(s-\tau_2)|^{4}ds
+\int_{t\wedge \tau_1}^{t\wedge \tau_2} |\varphi(0)-\varphi(s-\tau_2)|^{4}ds\right)^{1/2}.
\end{align*}
For the second addend, it follows from (\ref{lem2.1}) that
\begin{align*}
E\big|&\int_{\theta+\tau_{1}}^{\theta+\tau_{2}}b'_{3}(s,X_{\tau_{1}}(s), X_{\tau_{1}}(s-\tau_{1}))D_{\theta}X_{\tau_{1}}(s-\tau_{1})ds\big|^2\\
&\leq |\tau_1-\tau_2|\int_{\theta+\tau_{1}}^{\theta+\tau_{2}}E|b'_{3}(s,X_{\tau_{1}}(s), X_{\tau_{1}}(s-\tau_{1}))D_{\theta}X_{\tau_{1}}(s-\tau_{1})|^2ds\\
&\leq C_{L,T}|\tau_1-\tau_2|^2\leq C_{L,T}|\tau_1-\tau_2|.
\end{align*}
For the third addend, we use the H\"older inequality and (\ref{lem2.2}) to get
\begin{align*}
E\big|&\int_{\theta+\tau_{2}}^{t}b'_{3}(s,X_{\tau_{2}}(s), X_{\tau_{2}}(s-\tau_{2}))[D_{\theta}X_{\tau_{1}}(s-\tau_{1})-D_{\theta}X_{\tau_{2}}(s-\tau_{2})]ds\big|^2\\
&\leq C_{L,T}\int_{\theta+\tau_{2}}^{t}E|D_{\theta}X_{\tau_{1}}(s-\tau_{1})-D_{\theta}X_{\tau_{2}}(s-\tau_{2})|^2ds\\
&\leq C_{L,T}\int_{\theta+\tau_{2}}^{t}E|D_{\theta}X_{\tau_{1}}(s-\tau_{1})-D_{\theta}X_{\tau_{1}}(s-\tau_{2})|^2ds
+C_{L,T}\int_{\theta+\tau_{2}}^{t}E|D_{\theta}X_{\tau_{1}}(s-\tau_{2})-D_{\theta}X_{\tau_{2}}(s-\tau_{2})|^2ds\\
&\leq C_{L,T}\int_{\theta+\tau_{2}}^{t}|\tau_{1}-\tau_{2}|ds
+C_{L,T}\int_{\theta}^{t}E|D_{\theta}X_{\tau_{1}}(s)-D_{\theta}X_{\tau_{2}}(s)|^2ds\\
&\leq C_{L,T}|\tau_{1}-\tau_{2}|+C_{L,T}\int_{\theta}^{t}E|D_{\theta}X_{\tau_{1}}(s)-D_{\theta}X_{\tau_{2}}(s)|^2ds.
\end{align*}
Hence, we can obtain (\ref{hjf83}) for $E|I_{2,\tau_1}(\theta,t)-I_{2,\tau_2}(\theta,t)|^2.$ This finishes the proof of Lemma because the estimate for $E|J_{2,\tau_1}(\theta,t)-J_{2,\tau_2}(\theta,t)|^2$ can be treated similarly.
\end{proof}

\begin{prop}\label{lem3}
Let Assumptions \ref{assum1} and \ref{assum5} hold. Then, there exists a positive constants $C_{L,T}$ such that, for all $t\in[0,T],$
\begin{align}
E\|DX_{\tau_1}(t)&-DX_{\tau_2}(t)\|^2_{L^2[0,T]}\leq C_{T,L}\bigg(t^2|\tau_{1}-\tau_{2}|^{2}\notag\\
&+t\int_{0}^{t\wedge \tau_1} |\varphi(s-\tau_1)-\varphi(s-\tau_2)|^{4}ds
+t\int_{t\wedge \tau_1}^{t\wedge \tau_2} |\varphi(0)-\varphi(s-\tau_2)|^{4}ds\bigg)^{1/2},\label{lem3.1}
\end{align}
where $C_{L,T}$ is a positive constant.
\end{prop}
\begin{proof} We consider the following cases.

\noindent{\it Case 1: $t>\tau_2$.} In this case, we write
\begin{align*}
E\|D X_{\tau_{1}}(t)) - DX_{\tau_{2}}(t)\|^{2}_{L^2[0,T]}&=\int_0^{t-\tau_2}E|D_\theta X_{\tau_{1}}(t)) - D_{\theta}X_{\tau_{2}}(t)|^{2}d\theta+\int_{t-\tau_2}^{t-\tau_1}E|D_\theta X_{\tau_{1}}(t)) - D_{\theta}X_{\tau_{2}}(t)|^{2}d\theta\\&
+\int_{t-\tau_1}^{t}E|D_\theta X_{\tau_{1}}(t)) - D_{\theta}X_{\tau_{2}}(t)|^{2}d\theta.
\end{align*}
We observe from the equations \eqref{rem1.1} and \eqref{rem1.2} that, when $0\leq \theta\leq t-\tau_2,$ we have
\begin{align*}
D_\theta X_{\tau_{1}}(t)) - D_{\theta}X_{\tau_{2}}(t)&=\sigma(\theta, X_{\tau_{1}}(\theta),X_{\tau_{1}}(\theta-\tau_{1}))- \sigma(\theta, X_{\tau_{2}}(\theta),X_{\tau_{2}}(\theta-\tau_{2}))\\
&+I_{1,\tau_1}(\theta,t)-I_{1,\tau_2}(\theta,t)+I_{2,\tau_1}(\theta,t)-I_{2,\tau_2}(\theta,t)\\
&+J_{1,\tau_1}(\theta,t)-J_{1,\tau_2}(\theta,t)+J_{2,\tau_1}(\theta,t)-J_{2,\tau_2}(\theta,t).
\end{align*}
When $t-\tau_2< \theta\leq t-\tau_1,$ we have
\begin{align*}
D_\theta X_{\tau_{1}}(t)) - D_{\theta}X_{\tau_{2}}(t)&=\sigma(\theta, X_{\tau_{1}}(\theta),X_{\tau_{1}}(\theta-\tau_{1}))- \sigma(\theta, X_{\tau_{2}}(\theta),X_{\tau_{2}}(\theta-\tau_{2}))\\
&+I_{1,\tau_1}(\theta,t)-I_{1,\tau_2}(\theta,t)+I_{2,\tau_1}(\theta,t)\\
&+J_{1,\tau_1}(\theta,t)-J_{1,\tau_2}(\theta,t)+J_{2,\tau_1}(\theta,t).
\end{align*}
When $t-\tau_1< \theta\leq t,$ we have
\begin{align*}
D_\theta X_{\tau_{1}}(t) - D_{\theta}X_{\tau_{2}}(t)&=\sigma(\theta, X_{\tau_{1}}(\theta),X_{\tau_{1}}(\theta-\tau_{1}))- \sigma(\theta, X_{\tau_{2}}(\theta),X_{\tau_{2}}(\theta-\tau_{2}))\\
&+I_{1,\tau_1}(\theta,t)-I_{1,\tau_2}(\theta,t)+J_{1,\tau_1}(\theta,t)-J_{1,\tau_2}(\theta,t).
\end{align*}
Hence, using the estimates established in Lemmas \ref{jfks1} and \ref{jfks1a}, we deduce
\begin{align*}
E\|D &X_{\tau_{1}}(t)) - DX_{\tau_{2}}(t)\|^{2}_{L^2[0,T]}\\
&\leq \int_0^t E|\sigma(\theta, X_{\tau_{1}}(\theta),X_{\tau_{1}}(\theta-\tau_{1}))- \sigma(\theta, X_{\tau_{2}}(\theta),X_{\tau_{2}}(\theta-\tau_{2}))|^2d\theta\\
&+\int_{t-\tau_2}^{t-\tau_1}[E|I_{2,\tau_1}(\theta,t)|^{2}+E|J_{2,\tau_1}(\theta,t)|^{2}]d\theta\\
&+C_{L,T}\,t\left(|\tau_{1}-\tau_{2}|^{2}+\int_{0}^{t\wedge \tau_1} |\varphi(s-\tau_1)-\varphi(s-\tau_2)|^{4}ds
+\int_{t\wedge \tau_1}^{t\wedge \tau_2} |\varphi(0)-\varphi(s-\tau_2)|^{4}ds\right)^{1/2}\nonumber\\
&+C_{L,T}\int_{0}^{t}\int_{\theta}^{t} E|D_\theta X_{\tau_1}(s)-D_\theta X_{\tau_2}(s)|^2dsd\theta,\,\,\,t>\tau_2.
\end{align*}
By the H\"older inequality and Corollary \ref{imld5k}
 \begin{align}
 &\int_{0}^{t}E|\sigma(\theta, X_{\tau_{1}}(\theta), X_{\tau_{1}}(\theta-\tau_{1}))-\sigma(\theta, X_{\tau_{2}}(\theta), X_{\tau_{2}}(\theta-\tau_{2}))|^{2}d\theta \nonumber\\
 &\leq \sqrt{t}\left(\int_{0}^{t}E|\sigma(\theta, X_{\tau_{1}}(\theta), X_{\tau_{1}}(s-\theta))-\sigma(\theta, X_{\tau_{2}}(\theta), X_{\tau_{2}}(s-\theta))|^{4}d\theta\right)^{1/2}\nonumber\\
 &\leq C_{T,L}\left(t^2|\tau_{1}-\tau_{2}|^{2}+t\int_{0}^{t\wedge \tau_1} |\varphi(s-\tau_1)-\varphi(s-\tau_2)|^{4}ds
+t\int_{t\wedge \tau_1}^{t\wedge \tau_2} |\varphi(0)-\varphi(s-\tau_2)|^{4}ds\right)^{1/2}.\notag
 \end{align}
On the other hand, it is easy to see from the boundedness of $b_3',\sigma_3'$ and (\ref{lem2.1}) that $E|I_{2,\tau_1}(\theta,t)|^{2}+E|J_{2,\tau_1}(\theta,t)|^{2}\leq C_{L,T}(t-\theta-\tau_1)$ and hence,
$$\int_{t-\tau_2}^{t-\tau_1}[E|I_{2,\tau_1}(\theta,t)|^{2}+E|J_{2,\tau_1}(\theta,t)|^{2}]d\theta\leq C_{L,T}|\tau_{1}-\tau_{2}|^{2}\leq C_{L,T}\,t|\tau_{1}-\tau_{2}|.$$
Combining the above estimates yields
\begin{align*}
E\|D &X_{\tau_{1}}(t) - DX_{\tau_{2}}(t)\|^{2}_{L^2[0,T]}\\
&\leq C_{L,T}\left(t^2|\tau_{1}-\tau_{2}|^{2}+t\int_{0}^{t\wedge \tau_1} |\varphi(s-\tau_1)-\varphi(s-\tau_2)|^{4}ds
+t\int_{t\wedge \tau_1}^{t\wedge \tau_2} |\varphi(0)-\varphi(s-\tau_2)|^{4}ds\right)^{1/2}\nonumber\\
&+C_{L,T}\int_{0}^{t}\int_{\theta}^{t} E|D_\theta X_{\tau_1}(s)-D_\theta X_{\tau_2}(s)|^2dsd\theta\\
&= C_{L,T}\left(t^2|\tau_{1}-\tau_{2}|^{2}+t\int_{0}^{t\wedge \tau_1} |\varphi(s-\tau_1)-\varphi(s-\tau_2)|^{4}ds
+t\int_{t\wedge \tau_1}^{t\wedge \tau_2} |\varphi(0)-\varphi(s-\tau_2)|^{4}ds\right)^{1/2}\nonumber\\
&+C_{L,T}\int_{0}^{t} E\|DX_{\tau_{1}}(s) - DX_{\tau_{2}}(s)\|^{2}_{L^2[0,T]}ds,\,\,\,t>\tau_2,
\end{align*}
which, by  the Gronwall-type lemma (see, Theorem 1.4.2 in \cite{Pachpatte1998}), gives us (\ref{lem3.1}) when $t>\tau_2.$

We omit the detailed proof of {\it Case 2: $\tau_1<t\leq \tau_2$} and {\it Case 3: $t\leq \tau_1$} because the arguments of {\it Case 1} can be replicated.

The proof of Proposition is complete.
\end{proof}
We now are ready to state and prove the main result of this paper.
\begin{thm}\label{them2}
	Let Assumption \ref{assum1} and \ref{assum5} hold. We assume, in addition, that
	$$|\sigma(t,x,y)|\geq \sigma_0>0,\,\,\forall\,t\in [0,T],x,y\in \mathbb{R}.$$
Then, for any $g\in \mathcal{B}$  and  $t\in (0,T],$ we have
\begin{align}
|Eg(X_{\tau_{1}}(t))&-Eg(X_{\tau_{2}}(t))|\leq C_{T,L} \bigg(|\tau_{1}-\tau_{2}|^{2}\notag\\
 &+t^{-1}\int_{0}^{t\wedge \tau_1} |\varphi(s-\tau_1)-\varphi(s-\tau_2)|^{4}ds
+t^{-1}\int_{t\wedge \tau_1}^{t\wedge \tau_2} |\varphi(0)-\varphi(s-\tau_2)|^{4}ds\bigg)^{1/4},\label{0olkm3}
	\end{align}
where $C_{L,T}$ is a positive constant.
\end{thm}
\begin{proof} Fixed $t\in (0,T],$ we consider the random variables $F_1=X_{\tau_{1}}(t)$ and $F_2=X_{\tau_{2}}(t).$ Thanks to Propositions \ref{lem4}, \ref{9jk3k}, \ref{them1} and \ref{lem3} we have the following
$$E\|DF_1\|^{-8}_{L^2[0,T]}\leq C_{T,L}\,t^{-4},\,\,\,(E\|DF_1\|^{-2}_{L^2[0,T]})^2\leq C_{T,L}\,t^{-2},$$
\begin{align}
E\left(\int_0^T\int_0^T|D_\theta D_rF_1|^2d\theta dr\right)^2&=E\left(\int_0^t\int_0^t|D_\theta D_rE|X_{\tau_1}(t)|^2d\theta dr\right)^2\notag\\
&\leq t^2\int_0^t\int_0^tE|D_\theta D_rE|X_{\tau_1}(t)|^4d\theta dr\notag\\
&\leq C_{L,T}\,t^4,
\end{align}
and
\begin{align}
&\|F_1-F_2\|_{1,2}=\left(E|X_{\tau_1}(t)-X_{\tau_2}(t)|^2+E\|DX_{\tau_1}(t)-DX_{\tau_2}(t)\|^2_{L^2[0,T]}\right)^{1/2}\notag\\
&\leq C_{T,L}\left(t^2|\tau_{1}-\tau_{2}|^{2}+t\int_{0}^{t\wedge \tau_1} |\varphi(s-\tau_1)-\varphi(s-\tau_2)|^{4}ds
+t\int_{t\wedge \tau_1}^{t\wedge \tau_2} |\varphi(0)-\varphi(s-\tau_2)|^{4}ds\right)^{1/4}.\notag
\end{align}
Consequently, in view of Lemma \ref{lm2.1}, we obtain
\begin{align}
	&|Eg(X_{\tau_{1}}(t))-Eg(X_{\tau_{2}}(t))|=|Eg(F_1)-Eg(F_2)|\nonumber\\
&\leq C \left(E\|DF_1\|^{-8}_{L^2[0,T]}E\left(\int_0^T\int_0^T|D_\theta D_rF_1|^2d\theta dr\right)^2+(E\|DF_1\|^{-2}_{L^2[0,T]})^2\right)^{\frac{1}{4}}\|F_1-F_2\|_{1,2}\notag\\
&\leq C_{T,L}\left(|\tau_{1}-\tau_{2}|^{2}+t^{-1}\int_{0}^{t\wedge \tau_1} |\varphi(s-\tau_1)-\varphi(s-\tau_2)|^{4}ds
+t^{-1}\int_{t\wedge \tau_1}^{t\wedge \tau_2} |\varphi(0)-\varphi(s-\tau_2)|^{4}ds\right)^{1/4}.\notag
\end{align}
This completes the proof.
\end{proof}
%%%%%%%%%%%%%%%%%%%%%%%%%%%%%%%%%%%%%%%%%%%%%%%%%%%%%%%%%%%%%%%%%%%%%%%%%%%%%%%%%%%%%%%%%%%%%%%%%%%%%%%%%%%%%%%%%%%%%%%%%%%%%%%%%%%%%%%%%%%%%%%%%%%%%%%%%%%%%%%%%%%%%%%%%%%%%%%%%%%%%%%%%%%%%%%%%%%%

Clearly, when the initial data $\varphi$ is a continuous function, the estimate (\ref{0olkm3}) implies that $X_{\tau_{2}}(t)$ weakly converges to $X_{\tau_{1}}(t)$ as $\tau_2\to\tau_1.$ Moreover, when $\varphi$ is  H\"{o}lder continuous, we have the following H\"older continuity of solutions with respect to delay parameter.
\begin{cor}Suppose the assumptions of Theorem \ref{them2}. In addition, we assume that the initial data $\varphi$ is  H\"{o}lder continuous with exponent $\beta\in(0,1).$ Then, for any $g\in \mathcal{B},$   we have
\begin{equation}\label{jkfm5}
\sup\limits_{0\leq t\leq T}|Eg(X_{\tau_{1}}(t))-Eg(X_{\tau_{2}}(t))|\leq C_{T,L}\, |\tau_{1}-\tau_{2}|^{\beta\wedge\frac{1}{2}},
\end{equation}
where $C_{L,T}$ is a positive constant.
\end{cor}
\begin{proof}We have
$$t^{-1}\int_{t\wedge \tau_1}^{t\wedge \tau_2} (\tau_2-s)^{4\beta}ds\leq t^{-1}\int_{t\wedge \tau_1}^{t\wedge \tau_2} (\tau_2-\tau_1)^{4\beta}ds\leq(\tau_2-\tau_1)^{4\beta}.$$
Hence, it follows from (\ref{0olkm3}) that
\begin{align*}
	&|Eg(X_{\tau_{1}}(t))-Eg(X_{\tau_{2}}(t))|\\
&\leq C_{T,L}\left(|\tau_{1}-\tau_{2}|^{2}+t^{-1}\int_{0}^{t\wedge \tau_1} |\tau_1-\tau_2|^{4\beta}ds
+t^{-1}\int_{t\wedge \tau_1}^{t\wedge \tau_2} (\tau_2-s)^{4\beta}ds\right)^{1/4}\\
&\leq C_{T,L}\, |\tau_{1}-\tau_{2}|^{\beta\wedge\frac{1}{2}},\,\,\,0<t\leq T.
\end{align*}
This, together with the fact that $|Eg(X_{\tau_{1}}(t))-Eg(X_{\tau_{2}}(t))|=0$ when $t=0,$ gives us (\ref{jkfm5}).
\end{proof}
When the coefficients $b(t,x,y),\sigma(t,x,y)$ do not depend on $x,$ the delays $\tau_1=0,\tau_2=1/n$ and the initial data $\varphi(t)=x_0,\,t\in [-1/n,0],$ we obtain the following weak rate of convergence for the Carath\'eodory approximation.
\begin{cor}\label{kgld1} Consider the Carath\'eodory approximation system (\ref{cara01})-(\ref{cara02}). Suppose that

\noindent (i) $b(t,x)$ and $\sigma(t,x)$  are Lipschitz in $x$ and satisfy linear growth,

\noindent (ii) $b(t,\cdot)$ and $\sigma(t,\cdot)$  are twice differentiable with bounded derivatives,

\noindent (iii) $|\sigma(t,x)|\geq \sigma_0>0$ for all $t\in [0,T]$ and $x\in \mathbb{R}.$

\noindent Then, for any $g\in \mathcal{B},$ we have
$$\sup\limits_{0\leq t\leq T}|Eg(x^n(t))-Eg(x(t))|\leq \frac{C_{L,T}}{\sqrt{n}},\,\,\,n\geq 1,$$
where $C_{L,T}$ is a positive constant.
\end{cor}
\section{Conclusion}
In this paper, we employed the technique of Malliavin calculus to study the weak convergence of delay SDEs. The interesting point of our results lies in the fact that we are able to provide an explicit estimate for the rate of convergence and we only require the test function $g$ to be bounded. Furthermore, the method introduced in the paper can be used to investigate the weak convergence and the Carath\'eodory approximation for more general equations such as SDEs with multiple delays
$$dX(t)= b(t,X(t),X(t-\tau_1),\cdots,X(t-\tau_k))dt+\sigma(t,X(t),X(t-\tau_1),\cdots,X(t-\tau_k))dB(t),\,\,\,t\in [0,T]$$
or  SDEs with variable delays
$$dX(t)= b(t,X(t),X(t-\tau(t)))dt+\sigma(t,X(t),X(t-\tau(t)))dB(t),\,\,\,t\in [0,T].$$
However, the computations will be more complex. Hence, in the present paper, we have chosen the simplest class of delay SDEs to illustrate the main features of the method rather than getting bogged down with complex notations.

\noindent {\bf Acknowledgments.} The authors would like to thank the anonymous referees for their valuable comments for improving the paper. N. T. Dung and H. T. P. Thao are supported by the Vietnam National University, Hanoi under grant number QG.20.21.  T. C. Son, N. V. Tan, T. M. Cuong and P. D. Tung are supported by Viet Nam National Foundation for Science and Technology Development (NAFOSTED) under grant number 101.03-2019.08. A part of this paper was done while the authors were visiting the Vietnam Institute for Advanced Study in Mathematics (VIASM). The authors would like to thank the VIASM for financial support and hospitality.

\end{document}